\documentclass[english,reqno]{amsart}

\usepackage{amsmath, appendix, ulem}
\usepackage[nobysame]{amsrefs}
\usepackage{amssymb, color}
\usepackage[margin=1in]{geometry}

\usepackage{mathrsfs}
\usepackage{graphicx}
\usepackage{subfig}
\usepackage{float}
\usepackage{epsf}
\usepackage{hyperref}
\usepackage{titletoc}

\graphicspath{{../Figures/}}

\numberwithin{equation}{section}

\newtheorem{lem}{Lemma}[section]
\newtheorem{thm}{Theorem}[section]
\newtheorem{proposition}[thm]{Proposition}

\theoremstyle{remark}
\newtheorem{rmk}{Remark}[section]
\newtheorem{definition}[thm]{Definition}

\renewcommand{\tilde}{\widetilde}

\renewcommand{\bar}{\overline}

\newcommand{\nn}{\nonumber}
\newcommand{\R}{{\mathbb R}}

\newcommand{\del}{\partial}

\newcommand{\dt}{ \, {\rm d} t}
\newcommand{\dx}{ \, {\rm d} x}
\newcommand{\dy}{ \, {\rm d} y}

\newcommand{\ds}{\, {\rm d} s}

\newcommand{\dmu}{\, {\rm d} \mu}

\newcommand{\dpi}{\, {\rm d} \pi}

\newcommand{\Eps}{\varepsilon}
\newcommand{\Ni}{\noindent}

\newcommand{\EE}{\mathbb{E}}
\newcommand{\RR}{\mathbb{R}}

\newcommand{\PP}{\mathbb{P}}

\newcommand{\la}{\langle}
\newcommand{\ra}{\rangle}

\newcommand{\abs}[1]{\left\lvert#1\right\rvert}
\newcommand{\norm}[1]{\left\lVert#1 \, \right\rVert}

\newcommand{\viint}[2]{\left\langle#1, \, #2 \,\right\rangle}
\newcommand{\vpran}[1]{\left(#1\right)}

\author{Razvan C. Fetecau} 
\address{Department of Mathematics, Simon Fraser University, 8888 University Dr., Burnaby, BC V5A 1S6, Canada}
\email{van@math.sfu.ca}
\author{Hui Huang}
\address{Department of Mathematics, Simon Fraser University, 8888 University Dr., Burnaby, BC V5A 1S6, Canada}
\email{hha101@sfu.ca}
\author{Weiran Sun}
\address{Department of Mathematics, Simon Fraser University, 8888 University Dr., Burnaby, BC V5A 1S6, Canada}
\email{weirans@sfu.ca}
\begin{document}
\title[Propagation of Chaos]{Propagation of chaos for the Keller-Segel equation over bounded domains}
\maketitle
\begin{abstract}
In this paper we rigorously justify the propagation of chaos for the parabolic-elliptic Keller-Segel equation  over bounded convex domains. The boundary condition under consideration is the no-flux condition. As intermediate steps, we establish the well-posedness of the associated stochastic equation as well as the well-posedness of the Keller-Segel equation for bounded weak solutions.  
\end{abstract}

\section{Introduction}
In biology, \textit{chemotaxis} describes the phenomenon that cells and organisms direct their movement in response to chemical gradients. This mechanism enables bacteria such as Escherichia coli (E. coli)  to swim toward the highest concentration of food  or flee from poisons. It also facilitates the movement of sperm towards the egg during fertilization.
%
%
%
Mathematically, one of the most classical models for studying chemotaxis is the Keller-Segel (KS)-type equations, introduced in \cite{keller1970initiation} to describe aggregation of the slime mold amoebae. 

There exists a vast literature on the well-posedness and mean-field limits of Keller-Segel type equations  posed over the whole space $\R^d$ (see for example \cite{bian2013dynamic,blanchet2008infinite,dolbeault2004optimal,jager1992explosions,garcia2017,RY,huilearning}). Realistic physical settings in biological experiments involve however bounded domains (for example, cells such as bacteria or slime mold are cultured in Petri dishes which have physical boundaries). This motivates the work in the present paper, which is devoted to a Keller-Segel type equation over bounded domains.


Our specific goal in this paper is to justify the mean-field limit leading to the parabolic-elliptic Keller-Segel model in a bounded convex domain $D$. Specifically, let $\rho_t$ be the macroscopic density and $F$ be the interaction force among particles. Then the KS equation under investigation has the form
\begin{align}\label{PDE}
\begin{cases}
\partial_t\rho_t =\nu\triangle \rho_t-\nabla \cdot[\rho_t F\ast \rho_t] \,, 
\quad & x\in D\,, \quad t>0\,,\\[2pt]
\rho_t(x)|_{t=0}=\rho_0(x)\,,
\quad &x\in D \,, \\[2pt]
\langle \nu\nabla\rho_t-\rho_t F\ast \rho_t,n\rangle=0 \,,
\quad &x \in \partial D \,,
\end{cases}
\end{align}
where $\viint{\cdot}{\cdot}$ denotes the dot product in $\R^d$ $(d\geq 2)$, $n$ is the normal direction to $\partial D$, and the convolution is defined as
\begin{align*}
F \ast \rho_t (x) = \int_{D} F(x - y) \rho_t (y) \dy \,,
\qquad
\text{for any $x \in D$} \,.
\end{align*}
The particular $F$ considered in this paper is the negative gradient of the Newtonian potential given by
\begin{align}\label{F}
F(x) = -\nabla\Phi(x)
=-\frac{C_*x}{|x|^d}\,,
\quad 
\forall~ x\in\mathbb{R}^d\backslash\{0\}\,, \quad d\geq2\,,
\end{align} 
with $C_*=\frac{\Gamma(d/2)}{2\pi^{d/2}}$ and
\begin{align}\label{potential}
\Phi(x)
=\begin{cases}
\frac{1}{2\pi}\ln|x| \,, & \text{ if } d=2\,,\\[3pt]
-\frac{C_d }{|x|^{d-2}} \,,  & \text{ if } d\geq3\,.
\end{cases}
\end{align}
Here, $C_d=\frac{1}{d(d-2)\alpha_d}$ with $\alpha_d=\frac{\pi^{d/2}}{\Gamma(d/2+1)}$ denoting the volume of $d$-dimensional unit ball.  
%
Equation~\eqref{PDE} is also commonly written as a system of equations for the density $\rho_t$ and the concentration $c = - \Phi \ast \rho_t$ of the chemo-attractant. The main mechanism driving the solutions of  \eqref{PDE} is the competition between the diffusion and the nonlocal aggregation. We refer the reader to the review paper \cite{horstmann20031970}  or Chapter 5 in \cite{perthame2006transport} for more discussions of this model.

The main purpose of this paper is to justify the propagation of chaos for the KS system \eqref{PDE}. To do this, we model the dynamics of the particles by using equations on three levels: a continuity equation on the macroscopic level \eqref{PDE},  a stochastic equation on the mesoscopic level  and a stochastic interacting particle system on the microscopic level. 
%
%
The methodology in this paper is to approach the PDE model \eqref{PDE} from its underlying microscopic stochastic particle system. Specifically, let $(\Omega,\mathcal{F},(\mathcal{F}_t)_{t\geq0},\PP)$ be a probability space endowed with a filtration $(\mathcal{F}_t)_{t\geq0}$ and $\{B_t^i\}_{i=1}^N$ be $N$ independent $d$-dimensional Brownian motions on this probability space. The system of interacting particles underlying the KS equation \eqref{PDE} is given by:
\begin{align}\label{particle}
\begin{cases}
{X}_t^i={X}^i_0+\frac{1}{N-1}\sum\limits_{j\neq i}^N \int_{0}^tF({X}_s^i-{X}_s^j) \ds+\sqrt{2\nu}{B}_t^i-R_t^i\,,\quad i=1,\cdots, N, \quad t>0\,,
\\[3pt]
R_t^i=\int_0^tn(X_s^i) \,{\rm d}|R^i|_s\,,\quad |R^i|_t=\int_0^t\textbf{1}_{\partial D}(X_s^i) \,{\rm d}|R^i|_s \,, 
\qquad
X^i_t \big|_{t = 0} = X^i_{0} \,,
\end{cases}
\end{align}
where $\{{X}^i_t\}_{i=1}^N$ are the trajectories of $N$ particles and $F$ given by \eqref{F} models pairwise interaction between the particles. We assume that the initial data $\{{X}^i_0\}_{i=1}^N$ are independently, identically distributed (i.i.d.) random variables with a common probability density function $\rho_0(x)$ and $\sqrt{2\nu}$ is a constant.
Also, $n(x)$ denotes the outward normal to $\partial D$ at the point $x\in\partial D$ and $R_t^i$ is a reflecting process associated to $X_t^i$  with a bounded total variation. Moreover, $|R|_t^i$ is the total variation of $R_t^i$ on $[0,t]$, namely
\begin{equation*}
|R|_t^i=\sup\sum\limits_k|R_{t_k}^i-R_{t_{k-1}}^i|\,,
\end{equation*}
where the supremum is taken over all partitions such that $0=t_0<t_1<\cdots<t_n=t$.

Solving the stochastic differential equation (SDE) \eqref{particle} is known as the Skorokhod problem. Skorokhod introduced this problem in 1961 \cite{skorokhod1961stochastic,skorokhod1962stochastic}, where he considered a SDE with a reflecting diffusion process on $D=(0,\infty)$. The multi-dimensional version of Skorokhod's problem was solved by Tanaka \cite{tanaka1979stochastic}, where the domain $D$ was assumed to be convex. In \cite{lions1984stochastic} Lions and Sznitman extended the result to a general domain satisfying certain admissibility conditions, where the admissibility roughly means that the domain can be approximated by smooth domains in a certain sense. Later, Saisho \cite{saisho1987stochastic} removed the admissibility condition of the domain by applying the techniques used in \cite{tanaka1979stochastic}.

The analysis of the limiting process of interacting particle systems as $N \to \infty$ is usually called the \textit{mean-field limit}. 
There is extensive literature \cite{spohn2004dynamics,carrillo2010particle,jabin2014review,jabin2017mean} on mean-field limits for various models when the spatial domain is the whole space $\R^d$. For second order systems,  mean-field limits with globally Lipschitz forces in the whole space $\RR^d$ were obtained by Braun and Hepp \cite{braun1977vlasov} and Dobrushin \cite{dobrushin1979vlasov}. Later, in \cite{bolley2011stochastic},  these results were extended to particle systems with locally Lipschitz interacting forces under certain moment control assumptions. The whole space case with a singular force was treated in \cite{carrillo2018propagation,jabin2015particles,boers2016mean,lazarovici2015mean}. 
In particular, the mean-field limits for the  KS system \eqref{PDE} in the whole space were shown in \cite{fournier2015stochastic,HH1,HH2,garcia2017,RY,huilearning}. 

Meanwhile, there is much less work on mean-field limits for systems over bounded domains. In \cite{sznitman1984nonlinear}, the author derived rigorously the mean-field limit for particle systems with reflecting boundary conditions and bounded Lipschitz interaction forces. 
In Carrillo $et\,al.$ \cite{carrillo2014nonlocal}, the authors considered a model in general domains (which is not necessarily convex ) with a $\lambda$-geodesically convex potential. 
Very recently Choi and Salem \cite{choi2016propagation} proved the mean-field limit for the case where the particle system has a particular type of bounded discontinuous interacting force. 

The intermediate model between the microscopic model \eqref{particle} and the macroscopic equation \eqref{PDE} is the following stochastic equation for the mean-field self-consistent stochastic process  $({Y}_t)_{t\geq0}$:
\begin{equation}\label{SDE'}
\left\{\begin{array}{l}
{Y}_t={Y}_0+\int_{0}^t\int_{D}F({Y}_s-y)\rho_s(y) \dy\ds+\sqrt{2\nu}{B}_t-\tilde R_t\,, \quad t>0\,,\\ [3pt]
\tilde R_t=\int_0^tn(Y_s) \, {\rm d}|\tilde R|_s\,,\quad |\tilde R|_t=\int_0^t\textbf{1}_{\partial D}(Y_s) \, {\rm d}|\tilde R|_s\,,\\
\end{array}\right.
\end{equation}
where $(\rho_t)_{t\geq0}$ is the marginal density of $({Y}_t)_{t\geq0}$ for any $t\geq 0$ and $Y_0$ has density $\rho_0(x)$. The solution of \eqref{SDE'} is considered in the strong sense when the Brownian motion $B_t$ is prescribed.  By It\^{o}'s formula \cite{brent}, it can be shown that the probability density function $\rho_t$ of $Y_t$ satisfies the KS equation \eqref{PDE} in the weak sense (see Definition \ref{weaksolutionpde} below).

The goal of the present paper is to justify the mean-field limit from the particle system \eqref{particle} with the Newtonian force \eqref{F} to the mean-field KS equation \eqref{PDE} (see Theorem \ref{mainthm2}).  
More precisely, we show an explicit convergence rate from the weak solution of the mean-field equation \eqref{PDE} to the empirical measure associated with the particle system \eqref{particle}. The main idea is to introduce an approximate mean-field dynamics $\{Y_t^i\}_{i=1}^N$ satisfying \eqref{SDE'} and show that it well approximates the interacting particles $\{X_t^i\}_{i=1}^N$ with a regularized $F$ for large $N$'s. Our main contribution is to extend the propagation of chaos posed on the whole space with a singular force \cite{HH1,HH2,garcia2017,RY,huilearning} or posed on bounded domains with bounded forces \cite{choi2016propagation} to bounded convex domains with a singular force. We also show a universal upper bound of the expectation of the collision time in $\R^2$  (see Theorem \ref{collison}).

Before proving the propagation of chaos,  we need to establish the well-posedness of solutions for the KS equation \eqref{PDE} and the stochastic equation \eqref{SDE'}. 
There is a huge literature on the well-posedness of the KS equation (see for example \cite{bian2013dynamic,blanchet2008infinite,dolbeault2004optimal,jager1992explosions} for the existence and to \cite{liu2016refined,kawakami2016uniqueness,sugiyama2010uniqueness, carrillo2012uniqueness,fernandez2016uniqueness} for the uniqueness).  There are also results for variants of the KS equation set on bounded domains with no-flux boundary conditions (see for example \cite{biler1995existence,biler1998local,souplet2018blow,tao2012eventual,tao2011boundedness,winkler2016two,cieslak2017global,nagai2001blowup,souplet2018blow,jager1992explosions,bertozzi2009existence,bedrossian2011local}).
In this paper, we use a coupling method to justify the well-posedness of equations~\eqref{PDE} and~\eqref{SDE'}. This coupling method has been used to show the well-posedness for the KS equation in the whole space in \cite{RY,huang2016well}. Its main idea is to make use of the link between the KS equation \eqref{PDE} and the stochastic equation \eqref{SDE'} driven by a Brownian motion. Our main contribution in this part is to extend such coupling method from $\R^d$ to the bounded domain case.




The layout of the paper is as follows: in Section~\ref{sec:settings} we explain in detail the settings of our equations and the precise statements of the main theorems. In Section~\ref{sec:prelim}, we prove some technical lemmas that are used in the proofs. In Section~\ref{sec:wellposed}, we establish existence, uniqueness, and stability of the KS equation~\eqref{PDE} and the stochastic equation~\eqref{SDE'}. Finally, we show the finite-time collision and propagation of chaos in Section~\ref{sec:prop-chaos}.

\section{Settings and main results} \label{sec:settings}
In this section we explain the basic setting and main results in this paper.
Our main approach in investigating the mean-field limit with a singular potential is to apply a regularization mechanism \cite{HH2, HH1}.  To this end, let $J(x)$ be a blob function such that
\begin{align*}
J \in C^\infty(\mathbb{R}^d) \,,
\quad
J \geq 0 \,,
\quad
\mbox{supp}\,J(x)\subset {B}(0, 1)\,,
\quad
\int_{{B}(0, 1)}J(x) \dx=1.
\end{align*}
For any $\Eps > 0$, define $J_\varepsilon(x)=\frac{1}{\varepsilon^d}J(\frac{x}{\varepsilon})$ and
\begin{align} \label{def:F-Eps}
F_\varepsilon(x)=J_\varepsilon\ast F(x) \,.
\end{align}
The regularized version of \eqref{particle} has the form
\begin{align}\label{Rparticle}
\begin{cases}
{X}_t^{i,\varepsilon}={X}^i_0+\frac{1}{N-1}\sum\limits_{j\neq i}^N \int_{0}^tF_\varepsilon({X}_s^{i,\varepsilon}-{X}_s^{j,\varepsilon}) \ds+\sqrt{2\nu}{B}_t^{i,\varepsilon}-R_t^{i,\varepsilon}\,,\quad i=1,\cdots, N, \quad t>0\,, \\[3pt]
R_t^{i,\varepsilon}=\int_0^tn(X_t^{i,\varepsilon}) \, {\rm d} |R^{i,\varepsilon}|_s\,,\quad |R^{i,\varepsilon}|_t=\int_0^t\textbf{1}_{\partial D}(X_s^{i,\varepsilon}) \, {\rm d}|R^{i,\varepsilon}|_s \,,
\end{cases}
\end{align}
Here we assume that the initial data $\{{X}^i_0\}_{i=1}^N$ are i.i.d. random variables with a common probability density function $\rho_0(x)$.

The global well-posedness of \eqref{Rparticle} is shown in \cite[Theroem 1.1]{choi2016propagation}, which states 
\begin{thm} [\cite{choi2016propagation}]\label{choithm}
	Let $N\geq 2$. For any initial data $(X_0^i)_{i=1,\cdots,N}\in \overline{D}^N$ and for any $T>0$, there exists at least one weak solution $(X_t^{i,\varepsilon},R_t^{i,\varepsilon},{B}_t^{i,\varepsilon})_{i=1,\cdots,N}$ to system \eqref{Rparticle} on the time interval $[0,T]$.
\end{thm}

\begin{rmk}
	Note that the regularized system~\eqref{Rparticle} is the same as the original system~\eqref{particle} before the collision time $\tau$ defined in Theorem~\ref{collison}.
\end{rmk}

Next, we give the definition of weak solutions to the KS equation \eqref{PDE}.
\begin{definition}\label{weaksolutionpde} \textbf{(Weak solutions)}
	Assume $D\subset\RR^d$ is a bounded convex domain with  boundary $\del D \in C^1$.
	Let $T > 0$ and $0\leq\rho_0(x)\in  L^\infty(D)$. A probability density function $\rho_t(x)\in L^\infty\left(0, T; L^\infty(D)\right)$ is a weak solution to \eqref{PDE} if it holds that
	\begin{align}\label{Sindensityfun}\nonumber
	\int_{D}\rho_t(x)\phi(x) \dx
	&=\int_{D}\rho_0(x)\phi(x) \dx+\nu\int_0^t\int_{D}  \Delta \phi(x)\rho_s(x) \dx\ds
	\\
	&\qquad
	+\int_0^t\int_{D}\rho_s(x)F\ast\rho_s(x)\cdot\nabla\phi(x) \dx\ds  \, ,
	\end{align}
	for any test function $\phi\in {C}^\infty(\bar D)$ satisfying $\langle \nabla\phi,n \rangle=0$ on $\partial D$.
\end{definition}

In our first main theorems we establish the well-posedness of \eqref{SDE'} and \eqref{PDE} using the coupling method.
\begin{thm}\label{mainthm}{\textbf{(Existence)}}
	Let $Y_0$ be a random variable with density $\rho_0\in L^\infty(D)$. Suppose the singular interaction force $F$ is given by \eqref{F}. Then there exists $T > 0$ such that equation \eqref{SDE'} has a strong solution $(Y_t,\tilde R_t)$ and equation~\eqref{PDE} with initial data $\rho_0$ has a weak solution $\rho_t\in L^\infty\left(0, T; L^\infty(D)\right)$. Moreover, $\rho_t$ is the density of $Y_t$ up to time $T$. 
\end{thm} 

One main tool we use in establishing the above well-posedeness result is the Wasserstein distance. For the convenience of the reader, we give a brief introduction on the topology of the $p$-Wasserstein space.  We refer readers to \cite{ambrosio2008gradient,villani2008optimal} for further background.  Let us denote all probability measures on $D$ by $\mathcal{P}(D)$, then for $p\geq 1$, consider the space
\begin{eqnarray}\label{p1d}
\mathcal{P}_p(D)=\left\{f\in \mathcal{P}(D):~\int_{D}|x|^p \, {\rm d}f(x)<\infty\right\},
\end{eqnarray}
on which we will define our distances. 

\begin{definition}\label{Wpdistance}
	\textbf{($p$-Wasserstein  distance)} Let $p\in\mathbb{N}$ and $f$, $g\in \mathcal{P}_p(D)$. The $p$-Wasserstein  distance between $f$ and $g$ is defined by
		\begin{align} \label{def:W-p}
		W_p(f,g)
		= \vpran{\inf_{\pi\in\Lambda(f,~g)}\Big\{\int_{D\times D}|x-y |^p \dpi(x,y)\Big\}}^{1/p}
		= \vpran{\inf_{X\sim f,Y\sim g}\Big\{\EE[|X-Y|^p]\Big\}}^{1/p}.
		\end{align}
		where $\Lambda(f,~g)$ is the set of joint probability measures on $D\times D$ with marginals $f$ and $g$ respectively and $(X,Y)$ are all possible couples of random variables with $f$ and $g$ as their respective distributions.  
\end{definition}

\begin{definition}\textbf{($\infty$-Wasserstein distance)} Let $f$, $g\in \mathcal{P}_\infty(D)$, then the  $\infty$-Wasserstein distance between $f$ and $g$ is defined by
		\begin{equation}\label{Winf}
		W_\infty(f,g):=\inf_{X\sim f,Y\sim g}\PP\mbox{-ess sup }|X-Y|=\inf_{\pi\in \Lambda(f, \, g)}\pi\mathop{\mbox{-ess sup }}\limits_{(x,y)\in D\times D}|x-y| \,,
		\end{equation}	
		where random variables $X$,  $Y$ and joint distribution $\Lambda(f,~g)$ are used in Definition \ref{Wpdistance}. Here
		\begin{align}\label{pess}
		\PP\mbox{-ess sup }|X-Y|:=\inf\big \{\lambda\geq 0: ~\PP(|X-Y|>\lambda)=0\big\}\,,
		\end{align}
		and
		\begin{align}\label{piess}
		\pi \mathop{\mbox{-ess sup }}\limits_{(x,y)\in D\times D} |x - y|:=\inf\big \{\lambda\geq 0: ~\pi(\{(x,y)\in D\times D: |x-y|> \lambda\})=0\big\}\,.
		\end{align}
\end{definition}
The definitions of $\PP$-ess sup and $\pi\mathop{\mbox{-ess sup }}$ will be used in the proof of Lemma \ref{QLe}. As a slight abuse of notation, we also use $W_p(\rho_1,\rho_2)$ to denote the Wasserstein distance  of two measures  $f,g$ with densities $\rho_1,\rho_2$ respectively. In \cite[Theorem 6.18]{villani2008optimal}, it has been shown that for any $1 \leq p <\infty$, the space $\mathcal{P}_p(D)$ endowed with the metric $W_p$ is a complete metric space.  

Using the Wasserstein distance, we will establish the Dobrushin's type stability for \eqref{SDE'} and \eqref{PDE}. The result is given by the following theorem.

\begin{thm}\label{mainthm1}{\textbf{(Stability and Uniqueness)}}
	\Ni (1) If $\{Y_t^i\}_{i=1,2}$ are two strong solutions to \eqref{SDE'} with initial data $\{Y_0^i\}_{i=1,2}$ and densities $\{\rho_t^i\}_{i=1,2}$ respectively. Then there exist some constants $C_1$, $C_2$ depending only on $D$, $T$, $\|\rho^1_t\|_{L^\infty\left(0, T; L^\infty(D)\right)}$ and $\|\rho^2_t\|_{L^\infty\left(0, T; L^\infty(D)\right)}$ such that 
	\begin{align}\label{stab1}
	\sup\limits_{t\in{[0,T]}}\PP\mbox{-ess sup }|{Y}^1_t-{Y}^2_t| 
	\leq 
	C_1 \max \left\{\PP\mbox{-ess sup }|{Y}^1_0-{Y}^2_0|,  \,\, (\PP\mbox{-ess sup }|{Y}^1_0-{Y}^2_0|)^{\exp(-C_2T)}\right\} \,.
	\end{align}
	
	\Ni (2) Let $\{\rho^i_t\}_{i=1,2} \in L^\infty\left(0, T; L^\infty(D)\right)$ be two weak solutions to \eqref{PDE} with initial data $\{\rho_0^i\}_{i=1,2}$ respectively.  Then there exist some constants $C_1$, $C_2$ depending only on $D$, $T$, $\|\rho^1_t\|_{L^\infty\left(0, T; L^\infty(D)\right)}$ and $\|\rho^2_t\|_{L^\infty\left(0, T; L^\infty(D)\right)}$ such that 
	\begin{align}\label{stab2}
	\sup\limits_{t\in{[0,T]}}W_\infty(\rho_t^1,\rho_t^2) 
	\leq 
	C_1 \max \left\{W_\infty(\rho_0^1,~\rho_0^2),  \,\, (W_\infty(\rho_0^1,~\rho_0^2))^{\exp(-C_2T)}\right\}.
	\end{align}
	Uniqueness then follows from the stability estimates.
\end{thm}
\begin{rmk}
	This Dobrushin's type stability has been obtained for the KS equation in whole space \cite{huang2016well,liu2017random}. In a domain with boundaries, one has to deal with an extra term resulting from the boundary effect. The key observation is that due to the convexity of the domain, this extra term possesses a good sign (see in \eqref{goodsign} and \eqref{goodsign1}), and thus it does not add essential difficulties to the stability estimates. Such observation has been used in earlier works \cite{choi2016propagation,carrillo2014nonlocal,fernandez2016uniqueness}.
	
\end{rmk}

Our second main theorem is to show the propagation of chaos and justify the mean-field limit of the regularized particle system \eqref{Rparticle}. First, we recall the definition of a chaotic sequence \cite{kac1956foundations,sznitman1991topics}:
\begin{definition}\textbf{($\rho$-chaotic)}
	Let $E$ be a Polish space and $\rho$ a probability measure on $E$. Let $\{X^i\}_{i=1}^N$ be a sequence of exchangeable random variables. Then $\{X^i\}_{i=1}^N$ is $\rho$-chaotic if 
	\begin{align} \label{def:rho-chaotic}
	\mu^N
	:=\frac{1}{N}\sum\limits_{i=1}^N\delta_{X^i}\rightarrow \rho \,\, \mbox{ as measures} \quad \text{when} \, \, N\rightarrow\infty \,,
	\end{align}
	that is, for any $\phi \in C_b(D)$, 
	\begin{align*}
	\int_D \phi \dmu^N \to \int_{D} \phi \rho \dx
	\qquad
	\text{in law} \,.
	\end{align*}
\end{definition}
A sufficient condition for the sequence $\{X^i\}_{i=1}^N$ to be $\rho$-chaotic is 
\begin{proposition}{\cite[Remark 1.4]{choi2016propagation}}
	Let $\rho \in \mathcal{P}_p(E)$ with $p \geq 1$. Let $\mu^N$ be the empirical measure defined in~\eqref{def:rho-chaotic}. If 
	\begin{equation*}
	\EE\left[W_p(\mu^N,\rho)\right]\rightarrow0\mbox{ as }N\rightarrow\infty\,,
	\end{equation*}
	then the sequence $\{X^i\}_{i=1}^N$ is $\rho$-chaotic.
\end{proposition}

The result of propagation of chaos states
\begin{thm}\label{mainthm2}{\textbf{(Propagation of chaos)}}
	Let $\rho_t \in L^\infty\left(0, T; L^\infty(D)\right)$ be a weak solution to equation \eqref{PDE} with initial data $\rho_0\in L^\infty(D)$. Assume that $\{X_0^i\}_{i=1}^N$ are $N$ i.i.d. random variables with the common density $\rho_0$ and $\{X_t^{i,\varepsilon}\}_{i=1}^N$ satisfy the particle system \eqref{Rparticle} with the initial data $\{X_0^i\}_{i=1}^N$. Let $\mu^{X, \Eps}_t$ be the associated empirical measure given by
	\begin{equation*}
	\mu_t^{X,\varepsilon}:=\frac{1}{N}\sum\limits_{i=1}^N\delta_{X_t^{i,\varepsilon}} \,.
	\end{equation*}
	Then there exist a cut-off parameter $\varepsilon\sim \log^{-\frac{1}{d}}(N)$ and  constants $C_1,C_2, N_0>0$ depending only on $\norm{\rho_0}_{L^\infty(D)}$, $D$ and $T$ such that
	\begin{align*}
	\sup\limits_{t\in[0,T]}\EE\left[W_2(\mu_t^{X,\varepsilon},\rho_t)\right]
	\leq C_1[\omega(\log^{-\frac{1}{d}}(N))]^{\exp(-C_2T)}
	\qquad
	\text{ for all $N\geq N_0$}\,,
	\end{align*}
	where $\omega(x)$ is defined in~\eqref{Lipfunction1}.
\end{thm}

\begin{rmk}
	In the theorem above, we assume that the initial data $\{X_0^i\}_{i=1}^N$ are i.i.d. random variables with the common density $\rho_0$. This is a stronger assumption than $\{X_0^i\}^N_{i=1}$ being $\rho_0$-chaotic. Our theorem shows that under this stronger assumption on the initial data, the system remains chaotic for a finite time. 
\end{rmk}

\section{Preliminaries} \label{sec:prelim}
In this section we collect the technical lemmas that are used in the proofs of the main theorems. Throughout this paper, we will denote any generic constants as $C$, which may change from line to line. 
The notation $\norm{f}_{L^p(D)}$ represents the usual $L^p(D)$-norm for any $1\leq p\leq \infty$. We may suppress the dependence on $D$ when there is no confusion.  

The first lemma contains some useful properties of the regularized force $F_\Eps$, which states
\begin{lem} \cite{RY}\label{lmkenerl} 
	Let $F_\Eps$ be the regularized force defined in~\eqref{def:F-Eps}. Then
	\\
	\Ni (a)
	$F_\varepsilon(-x)=-F_\varepsilon(x)$, $F_\varepsilon(x)=F(x)$ for any $|x|\geq \varepsilon$, and $|F_\varepsilon(x)|\leq |F(x)|$ for any $x \in \R^d$.
	
	\Ni (b)
	$|\partial^\beta F_\varepsilon(x)|\leq C_\beta\,\varepsilon^{1-d-|\beta|}$ for any $x\in\mathbb{R}^d$ and $|\beta| \geq 0$.
	
\end{lem}

An immediate consequence of Lemma~\ref{lmkenerl} is
\begin{lem} \label{L1}
	For any function $\rho \in  L^\infty(D)$, there exists a universal constant $C$ that only depends on $D$ such that for all $\varepsilon\geq\varepsilon^{'}\geq0$, 
	\\
	\Ni (a)
	$\norm{F_\varepsilon\ast\rho}_{L^\infty(D)} \leq C\|\rho\|_{L^\infty(D)}$.
	
	\Ni (b)
	$\norm{F_\varepsilon\ast\rho-F_{\varepsilon'}\ast\rho}_{L^\infty(D)}  
	\leq C \varepsilon \|\rho\|_{L^\infty(D)}$.
\end{lem}
\begin{proof}
	(a). By Lemma~\ref{lmkenerl} (a), we have
	\begin{align*}
	\abs{F_\varepsilon\ast\rho}
	\leq 
	\vpran{\int_{D} F_\Eps (x - y) \dy} \norm{\rho}_{L^\infty(D)}
	\leq
	\vpran{\int_{D} \frac{C_\ast}{|x - y|^{d-1}} \dy} \norm{\rho}_{L^\infty(D)}
	\leq
	C \norm{\rho}_{L^\infty(D)} \,.
	\end{align*}
	
	
	\Ni (b). By Lemma \ref{lmkenerl} (a) again, 
	we have
	\begin{align*}
	|F_\varepsilon\ast\rho(x)-F_{\varepsilon'}\ast\rho(x)|
	&\leq 
	\int_{D}\big|F_\varepsilon\big(x-y\big)-F_{\varepsilon^{'}}\big(x-y\big)\big|\rho(y) \dy
	\\
	&\leq C\|\rho\|_{L^\infty(D)}\int_{\{y:|x-y|\leq\varepsilon \}}\frac{\dy}{|x-y|^{d-1}}
	\leq C\, \varepsilon \|\rho\|_{L^\infty(D)}\,,
	\end{align*}
	where $C$ is a generic constant. 
\end{proof}

Using the above two lemmas, we can prove a crucial Quasi-Log-Lipschitz estimate for $F_\Eps$. 
\begin{lem}[Quasi-log-Lipschitz]\label{QLe} 
	Let $X_1, X_2$ be two random variables with densities $\rho_1, \rho_2 \in  L^\infty(D)$ respectively (${X}_1$  and ${X}_2$ may not be independent). Suppose $r_0$ is large enough such that $D \subseteq B(0, r_0)$. Then for any $0 \leq \varepsilon' \leq \varepsilon < 1$, there exists a constant $C$ depending only on $D$, $\|\rho_1\|_{ L^\infty}$ and $\|\rho_2\|_{ L^\infty}$ such that
	\begin{align}\label{distantest}
	\abs{F_\varepsilon\ast\rho_1(X_1)-F_{\varepsilon'}\ast\rho_2(X_2)}
	\leq 
	C \omega(\Eps)
	+ C\omega\vpran{\PP\mbox{-ess sup }|X_1-X_2|} \,,
	\end{align}
	where
	\begin{align}\label{Lipfunction1}
	\omega(r)
	=\begin{cases}
	1 \,,  & \text{if  $r\geq 1$} \,,\\[2pt]
	r(1-\log r) \,,  & \text{if $0<r< 1$} \,.
	\end{cases}
	\end{align}
\end{lem}
\begin{proof}
	Splitting the difference into two parts such that
	\begin{align}\label{zongest1}\nonumber
	F_\varepsilon\ast\rho_1(X_1)-F_{\varepsilon'}\ast\rho_2(X_2)
	&= \vpran{F_\varepsilon\ast\rho_1(X_1) - F_\varepsilon\ast\rho_2(X_2)}
	+ \vpran{F_\varepsilon\ast\rho_2(X_2)-F_{\varepsilon'}\ast\rho_2(X_2)}
	\\
	&=:I_1+I_2 \,,
	\end{align}
	we will derive the desired bound by bounding $I_1, I_2$ separately. The bound of $I_2$ is readily obtained by Lemma \ref{L1} $(b)$, which shows  that there exists a constant $C$ depending only on $\|{\rho}_2\|_ {L^\infty}$ such that
	\begin{eqnarray}\label{nonlinearestimates1}
	|I_2|\leq C\,\varepsilon\,.
	\end{eqnarray}
	
	The estimate of $I_1$ can	be obtained by following the approach in \cite[Lemma 1.4]{kato1967classical}.
	To this end, we introduce the joint distribution of $X_1$ and $X_2$ as $\pi:=\mathcal{L}(X_1,X_2)$. Then 
	\begin{align*}
	|I_1|= \left|\iint_{D\times D} F_\varepsilon(X_1-x_1)-F_{\varepsilon}(X_2-x_2)\pi(\dx_1,\dx_2)\right|.
	\end{align*}
	Let $\ell$ be the constant defined by
	\begin{align*}
	\ell:
	=\PP\mbox{-ess sup }|X_1-X_2|
	+\pi \mathop{\mbox{-ess sup }}\limits_{(x_1,x_2)\in D\times D} |x_1 - x_2| \,.
	\end{align*} 
	By their definitions in \eqref{pess} and \eqref{piess}, we have
	\begin{align*}
	\PP\mbox{-ess sup }|X_1-X_2|
	=\pi\mathop{\mbox{-ess sup }}\limits_{(x_1,x_2)\in D\times D} |x_1 - x_2| \,.
	\end{align*}
	Splitting $|I_1|$ into two parts, we have
	\begin{align*}
	& \quad \,
	\abs{\iint_{D\times D} F_\varepsilon(X_1-x_1)-F_{\varepsilon}(X_2-x_2)\pi(\dx_1, \dx_2)} \nn
	\\
	& \leq \abs{\iint_{(\{|X_1-x_1|\leq 2\ell\}\cap D)\times D} F_\varepsilon(X_1-x_1)-F_{\varepsilon}(X_2-x_2)\pi(\dx_1, \dx_2)} \nn 
	\\
	& \quad \,
	+ \abs{\iint_{ (\{|X_1-x_1|> 2\ell\}\cap D)\times D} F_\varepsilon(X_1-x_1)-F_{\varepsilon}(X_2-x_2)\pi(\dx_1, \dx_2)}  \nn\\
	& =: I_{11} + I_{12} \,.
	\end{align*}
	Note that over the integration domain of $I_{11}$, we have
	\begin{align*}
	|X_2-x_2|
	&\leq 
	|X_1-x_1|+|X_2-X_1+x_1-x_2| \nn 
	\\
	&\leq 
	2\ell+\PP\mbox{-ess sup }|X_1-X_2|
	+ \pi\mathop{\mbox{-ess sup }}\limits_{(x,y)\in D\times D} |x_1- x_2|
	= 3\ell \,.
	\end{align*}
	Therefore, $I_{11}$ satisfies
	\begin{align}\label{I11}
	I_{11}
	&\leq  
	\int_{|X_1-x_1|\leq 2\ell} \frac{C}{|X_1-x_1|^{d-1}}\rho_1(x_1)\dx_1
	+ \int_{|X_2-x_2|\leq 3\ell}\frac{C}{|X_2-x_2|^{d-1}}\rho_2(x_2)\dx_2 \nn
	\\
	&\leq 
	C(\norm{\rho_1}_\infty+\norm{\rho_2}_{L^\infty(D)}) \ell \,,
	\end{align}
	where we have applied the bound $|F_\varepsilon(x)|\leq |F(x)|$.
	
	To bound $|I_{12}|$, we separate the two cases where $2 \ell > 2\Eps$ and $2 \ell \leq 2\Eps$. First, if $2\ell > 2\Eps$, then $\abs{X_1 - x_1} > 2 \ell > 2\Eps $. By the definition of $\ell$ we have
	\begin{align*}
	\ell\geq |X_1-x_1|-|X_2-x_2| >2\ell- |X_2-x_2|\,,
	\end{align*}
	which implies $ |X_2-x_2|>\ell>2\varepsilon$.
	Therefore, if $2 \ell > \Eps$, then
	\begin{align*}
	F_\varepsilon(X_1-x_1)
	= F(X_1-x_1),\quad F_\varepsilon(X_2-x_2)=F(X_2-x_2) \,,
	\end{align*}
	and
	\begin{align*}
	F_\varepsilon(X_1-x_1) - F_\varepsilon(X_2-x_2)
	\leq \frac{C|X_2-X_1+x_1-x_2|}{|X_1-x_1+\xi(X_2-X_1+x_1-x_2)|^{d}} 
	\qquad
	\text{for some $\xi \in [0, 1]$}\,.
	\end{align*}
	Note that
	\begin{align*}
	|X_1-x_1+\xi(X_2-X_1+x_1-x_2)|  
	\geq& |X_1-x_1|-|X_2-X_1+x_1-x_2|
	\\
	\geq &|X_1-x_1|-\ell \geq |X_1-x_1|- \frac{1}{2} |X_1-x_1|=\frac{1}{2} |X_1-x_1|\,.
	\end{align*}
	Therefore,
	\begin{align}\label{I12}
	|I_{12}| 
	&\leq 
	\iint_{\{2\ell< |X_1-x_1| \leq r_0\}\times D} |X_2-X_1+x_1-x_2|
	\frac{C}{|X_1-x_1+\xi(X_2-X_1+x_1-x_2)|^{d}}\pi(\dx_1,\dx_2)  \nn
	\\
	&\leq C\ell  \int_{2\ell <|X_1-x_1|\leq r_0} \frac{2^d}{|X_1-x_1|^{d}}\rho_1(x_1)\dx_1
	\leq C\norm{\rho_1}_{L^\infty}\ell\log(\frac{r_0}{2\ell})\,,
	\qquad
	\text{if $2 \ell > 2\Eps$\,.}
	\end{align}
	Thus it follows from \eqref{I11} and \eqref{I12} that
	\begin{align}\label{I1case1}
	|I_1|
	\leq C \ell (1-\log(\ell))\,,
	\quad 
	\text{if $\varepsilon<\ell<1$}\,.
	\end{align}

	In the case where $2 \ell \leq 2\Eps$, we further divide $I_{12}$ as
	\begin{align*}
	I_{12}
	&\leq 
	\abs{\iint_{ (\{|X_1-x_1|> 2\Eps \}\cap D)\times D} F_\varepsilon(X_1-x_1)-F_{\varepsilon}(X_2-x_2)\pi(\dx_1, \dx_2)}
	\\
	& \quad \,
	+ \iint_{ (\{2 \ell \leq |X_1-x_1| \leq 2\Eps\}\cap D)\times D} \abs{F_\varepsilon(X_1-x_1)-F_{\varepsilon}(X_2-x_2)} \pi(\dx_1, \dx_2) 
	:= I^{(1)}_{12} + I^{(2)}_{12} \,.
	\end{align*}
	The first term $I^{(1)}_{12}$ is bounded in a similar as in the case $2 \ell > 2\Eps$, which gives
	\begin{align} \label{bound:I-1-12}
	I^{(1)}_{12} 
	\leq 
	C \norm{\rho_1}_{L^\infty} \Eps \log(\frac{r_0}{2\Eps}) \,.
	\end{align}
	By Lemma~\ref{lmkenerl} (b), the second term $I^{(2)}_{12}$ satisfies
	\begin{align} \label{bound:I-2-12}
	I^{(2)}_{12}
	& \leq 
	\iint_{ (|X_1-x_1| \leq 2\Eps\}\cap D)\times D} 
	\abs{X_1 - X_2 - x_1 + x_2} \Eps^{-d} \pi(\dx_1, \dx_2)\notag
	\\
	& \leq
	C \ell \Eps^{-d} \int_{|X_1 - x_1| \leq 2 \Eps} \rho(x_1) \dx_1
	\leq
	C \ell  \,. 
	\end{align}
	Combining~\eqref{bound:I-1-12} and~\eqref{bound:I-2-12}, we have
	\begin{align} \label{bound:I-12-case-2}
	I_{12} \leq C \omega(\Eps) \,,
	\qquad  
	\text{if $2 \ell \leq 2\Eps$} \,.
	\end{align}
	Collecting estimates \eqref{I11} and \eqref{bound:I-12-case-2}, we derive that 
	\begin{align*}
	|I_1|
	\leq 
	C \omega(\Eps)\,,
	\quad 
	\text{if $0<\ell\leq \varepsilon$}\,.
	\end{align*}
	Combing this with \eqref{I1case1}, we concludes that
	\begin{align*}
	|I_1|
	\leq 
	C \omega(\Eps) +C\omega(\ell)\,,
	\quad 
	\text{if $0<\ell<1$}\,.
	\end{align*}
	where $C$ depends only on $D$, $\norm{\rho_1}_{L^\infty(D)}$ and $\norm{\rho_2}_{L^\infty(D)}$. For $\ell\geq 1$, it follows from Lemma \ref{L1} (i) that $|I_1|\leq C$. We thereby conclude our proof.
	%
	%
\end{proof}

There are two Gronwall-type inequalities used in the later proofs. The first one is
\begin{lem} \cite[Lemma A.1]{choi2016propagation} \label{L2-1}
	Let $f$ be a nonnegative scalar function satisfying
	\begin{align*}
	f(t)
	\leq 
	f_0\exp\left(C\int_0^t f(s)\ds\right),\quad t\geq 0 \,.
	\end{align*}
	Then 
	\begin{align*}
	f(t)
	\leq
	\frac{f_0}{1-Cf_0t}\,,\quad t\geq 0\,,
	\end{align*}
	where $C$ is a positive constant.
\end{lem}


The second Gronwall-type inequality states
\begin{lem} \label{L2}
	Let $\{f_\varepsilon(t)\}_{\varepsilon>0}$ be a family of nonnegative continuous functions  satisfying
	\begin{align} \label{ineq:f-Eps}
	f_\varepsilon^2(t)
	\leq 
	C\int_0^tf_\varepsilon(s)\omega(f_\varepsilon(s))\ds+C\omega(\varepsilon) T\quad\mbox{~for~all } t\in[0,T]\,,
	\end{align}
	where $C>0$ is a constant, $T$ is independent of $\Eps$, and $\omega(x)$ is defined in \eqref{Lipfunction1}. Suppose $\Eps$ is small enough such that $C\omega(\varepsilon) T < 1$.
	Then there exist two constants $C_T$ and $0 <\varLambda_0 < 1/2$ depending only on $T$ and $C$  such that if $\omega(\varepsilon)\leq \varLambda_0$, then
	\begin{align}\label{uniestimatet}
	\sup\limits_{t\in[0,T]}f_\varepsilon(t)
	\leq 
	C_T\omega(\varepsilon)^{\frac{1}{2}e^{-CT}}<1\,.
	\end{align}
\end{lem}
\begin{proof}
	Let $R_\Eps$ be the continuous function defined by
	\begin{align*}
	R_\varepsilon(t)
	:= C\omega(\varepsilon) T
	+\int_0^t Cf_\varepsilon(s)\,\omega(f_\varepsilon(s))\ds\,.
	\end{align*}
	Since $R_\varepsilon(0)<1$ and $R_\varepsilon(t)$ is continuous in time, there  exists $T^\ast>0$ such that $R_\varepsilon(t) <1$ for all $t\in[0,T^\ast)$. We will show that such $T^\ast$ can be extended to $T$ by choosing $\Eps$ small enough.
	Indeed, for $\omega(\varepsilon)< 1/2$, define
	\begin{align} \label{def:F-T-Eps}
	F(s)
	:= \exp(1-\exp(-Cs))\omega(\varepsilon)^{\frac{1}{2}\exp(-Cs)}
	\quad \text{and} \quad
	T_\varepsilon
	:= \frac{1}{C}\log(1-\frac{1}{2}\log(\omega(\varepsilon))) \,.
	\end{align}
	Then $F$ is an increasing function and the particular choice of $T_\Eps$ satisfies 
	\begin{align} \label{bound:T-Eps}
	F(T_\Eps) = 1 \,,
	\qquad
	T_\Eps 
	\geq 
	\frac{1}{C} \log\vpran{1 + \frac{1}{2}\log2} > 0 \,,
	\qquad
	\omega(\varepsilon) < 1/2 \,.
	\end{align}
	Now we show that $0 < R_\varepsilon(t)<1$ for all $t\in[0,T_\varepsilon)$. We prove this by contradiction. Assume there exists $T_1<T_\varepsilon$ such that \begin{align*}
	R_\varepsilon(t) < 1
	\qquad
	\text{for any $t \in [0, T_1)$}
	\quad \text{and} \quad
	R_\varepsilon(T_1) = 1 \,.
	\end{align*}
	Since $\omega(x) = x \vpran{1 - \log x}$ on $(0, 1]$ and it is nondecreasing, by~\eqref{ineq:f-Eps} we have
	\begin{align}\label{41}
	\frac{{\rm d} R_\varepsilon(t)}{\dt}
	= Cf_\varepsilon(t)\omega(f_\varepsilon(t))
	\leq 
	C(R_\varepsilon(t))^{1/2}\omega((R_\varepsilon(t))^{1/2}) \,,
	\end{align}
	which gives
	\begin{align*}
	\frac{\rm d}{\dt}(R_\varepsilon(t))^{\frac{1}{2}}
	= \frac{1}{2}(R_\varepsilon(t))^{-\frac{1}{2}}\frac{{\rm d} R_\varepsilon(t)}{\dt}
	\leq
	\frac{C}{2}\omega\left((R_\varepsilon(t))^{1/2}\right)\leq C(R_\varepsilon(t))^{1/2}\left(1-\log((R_\varepsilon(t))^{\frac{1}{2}})\right)\,.
	\end{align*}
	Solving this differential inequality on $t \in [0, T_1]$, we get
	\begin{equation}\label{Rt}
	(R_\varepsilon(t))^{\frac{1}{2}}
	\leq 
	e^{1-(1-\frac{1}{2}\log(C\omega(\varepsilon)T_1))e^{-CT_1}}\leq \exp(1-\exp(-CT_1))\omega(\varepsilon)^{\frac{1}{2}\exp(-CT_1)} = F(T_1)\,.
	\end{equation}
	By the definitions and properties of $F, T_1, T_\Eps$, we have
	\begin{align*}
	(R_\varepsilon(T_1))^{\frac{1}{2}}
	\leq F(T_1)
	<    F(T_\varepsilon) = 1\,,
	\end{align*}
	which contradicts the assumption $R_\varepsilon(T_1)=1$. 
	Hence we have proved that  $R_\varepsilon(t) < 1$ for all $t\in[0,T_\varepsilon)$ with $T_\Eps$ defined in~\eqref{def:F-T-Eps}.  Notice that by the definition of $T_\Eps$, there exists $\varLambda_0 < 1/2$ depending only on $C$ and $T$ such that $T_\varepsilon>T$ for all $\omega(\varepsilon)\leq \varLambda_0$. For such $\Eps$ we have 
	\begin{align*}
	R_\varepsilon(t) <1 \qquad \text{for all $t\in [0,T]$} \,.
	\end{align*} 
	Then repeating the derivation of~\eqref{Rt}, we obtain that
	\begin{align*}
	f_\varepsilon(t)
	\leq 
	(R_\varepsilon(t))^{\frac{1}{2}}
	\leq F(T) 
	\leq C_T\omega(\varepsilon)^{\frac{1}{2}e^{-CT}}
	\quad \mbox{for all }t\in[0,T] \,.
	\end{align*}
\end{proof}

\section{Well-posedness of the nonlinear SDE and PDE} \label{sec:wellposed}
In this section we establish the well-posedness of the nonlinear SDE \eqref{SDE'} and show that the marginal density of its solution is the unique weak solution to the KS equation \eqref{PDE}.
\subsection{Regularized system}
The existence of solutions to~\eqref{SDE'} and~\eqref{PDE} will be obtained through the convergence of approximate solutions to regularized versions of~\eqref{SDE'} and~\eqref{PDE}, where the singular force $F$ is replaced by $F_\Eps$ defined in~\eqref{def:F-Eps}. 


First, we show the well-posedness of the regularized PDE 
\begin{align}\label{RPDE}
\begin{cases}
\partial_t\rho_t^\varepsilon=\nu\triangle \rho_t^\varepsilon-\nabla \cdot[\rho_t ^\varepsilon F_\varepsilon\ast \rho_t^\varepsilon]\,, & x\in D\,, \quad t>0\,,
\\[2pt]
\rho_t^\varepsilon(x)|_{t=0}=\rho_0(x)\,, &x\in D \,, 
\\[2pt]
\langle\nu\nabla\rho_t^\varepsilon-\rho_t^\varepsilon F_\varepsilon\ast \rho_t^\varepsilon,n\rangle=0\,, &\mbox{ on } \partial D\,.
\end{cases}
\end{align}
\begin{proposition}\label{propRPDE}
	Suppose $\rho_0 \in L^\infty(D)$. Then there exists $T > 0$ independent of $\Eps$ such that equation~\eqref{RPDE} has a weak solution $\rho_t^\varepsilon \in L^\infty\left(0, T; L^\infty(D)\right)$ with the bound
	\begin{align} \label{bound:rho-Eps}
	\sup\limits_{0\leq t\leq T}\|\rho_t^\varepsilon\|_{L^\infty(D)}
	\leq 
	2 \norm{\rho_0}_{L^\infty(D)} \,.
	\end{align}
\end{proposition}

\begin{proof}
	In order to show the existence of a weak solution to \eqref{RPDE}, first we consider the regularized equation
	\begin{align}\label{RPDE1}
	\begin{cases}
	\partial_t\rho_t^{\varepsilon,\delta}
	=\nu\triangle \rho_t^{\varepsilon,\delta}
	- \nabla \cdot[\rho_t ^\varepsilon F_\varepsilon\ast \rho_t^{\varepsilon,\delta}] \,, & x\in D\,, \quad t>0\,,
	\\[2pt]
	\rho_t^{\varepsilon,\delta}(x)|_{t=0}=\rho_0^\delta(x) \,, & x\in D \,,
	\\[2pt]
	\langle\nu\nabla\rho_t^{\varepsilon,\delta}-\rho_t^{\varepsilon,\delta} F_\varepsilon\ast \rho_t^{\varepsilon,\delta},n\rangle=0 &\mbox{ on } \partial D\,,
	\end{cases}
	\end{align}
	where $\rho_0^\delta\in C^\infty(D)$ satisfies
	\begin{align*}
	\norm{\rho_0^\delta}_{L^\infty(D)}\leq \norm{\rho_0}_{L^\infty(D)} \,,
	\qquad
	\rho_0^\delta\rightarrow \rho_0\quad\mbox{in }L^1\mbox{ as }\delta\rightarrow 0\,.
	\end{align*}
	Then for any fixed $\Eps, \delta>0$, there exists $T_1$ (which may depend on $\Eps, \delta$) such that system \eqref{RPDE1} has a smooth nonnegative solution $\rho^{\Eps, \delta}_t \in C([0, T_1); L^\infty(D))$ c.f. \cite{carrillo2003kinetic}. In order to remove the dependence of $T_1$ on $\Eps, \delta$, we show (a priori) that there exist $T, C > 0$ which only depend on $\norm{\rho_0}_{L^\infty}$ and the domain $D$ such that
	\begin{equation*}
	\sup\limits_{t\in[0,T]}\norm{\rho_t^{\varepsilon,\delta}}_{L^\infty(D)}\leq C \,.
	\end{equation*}
	Indeed, for any $p \geq 2$,  we have
	\begin{align} \label{eq:rho-delta-p}
	\frac{{\rm d}\norm{\rho_t^{\varepsilon,\delta}}_{L^p(D)}^p}{\dt}
	&=p\int_D(\rho_t^{\varepsilon,\delta})^{p-1}(\nu\Delta\rho_t^{\varepsilon,\delta}-\nabla\cdot(\rho_t^{\varepsilon,\delta} F_\varepsilon\ast\rho_t^{\varepsilon,\delta})) \dx \nn
	\\
	&=p\nu\int_{\partial D}(\rho_t^{\varepsilon,\delta})^{p-1}\nabla\rho_t^{\varepsilon,\delta}\cdot n \,{\rm d}S(x)-p\nu \int_D\nabla\rho_t^{\varepsilon,\delta}\cdot \nabla ((\rho_t^{\varepsilon,\delta})^{p-1}) \dx \nn
	\\
	&\quad
	- p\int_{\partial D}(\rho_t^{\varepsilon,\delta})^{p} F_\varepsilon\ast\rho_t^{\varepsilon,\delta}\cdot n \,{\rm d}S(x)+p\int_D\rho_t^{\varepsilon,\delta} F_\varepsilon\ast\rho_t^{\varepsilon,\delta}\cdot\nabla((\rho_t^{\varepsilon,\delta})^{p-1}) \dx \nn
	\\
	&=-p(p-1)\nu \int_D(\rho_t^{\varepsilon,\delta})^{p-2}|\nabla \rho_t^{\varepsilon,\delta}|^2 \dx
	+p\int_D\rho_t^{\varepsilon,\delta} F_\varepsilon\ast\rho_t^{\varepsilon,\delta}\cdot\nabla((\rho_t^{\varepsilon,\delta})^{p-1}) \dx \,.
	\end{align}
	re-organize equation~\eqref{eq:rho-delta-p} as
	\begin{align} \label{eq:rho-delta-p-1}
	\frac{{\rm d}\norm{\rho_t^{\varepsilon,\delta}}_{L^p(D)}^p}{\dt}
	+\frac{4\nu(p-1)}{p}\norm{\nabla((\rho_t^{\varepsilon,\delta})^{\frac{p}{2}})}_{L^2(D)}^2
	&=(p-1)\int_D F_\varepsilon\ast\rho_t^{\varepsilon,\delta}\cdot\nabla((\rho_t^{\varepsilon,\delta})^{p}) \dx\notag 
	\\
	&=2(p-1)\int_D (\rho_t^{\varepsilon,\delta})^{\frac{p}{2}}F_\varepsilon\ast\rho_t^{\varepsilon,\delta}\cdot\nabla((\rho_t^{\varepsilon,\delta})^{\frac{p}{2}}) \dx\,.
	\end{align}
	By $(a)$ in Lemma \ref{L1}, one has
	\begin{equation*}
	\int_D (\rho_t^{\varepsilon,\delta})^{\frac{p}{2}}F_\varepsilon\ast\rho_t^{\varepsilon,\delta}\cdot\nabla((\rho_t^{\varepsilon,\delta})^{\frac{p}{2}}) \dx
	\leq 
	C\norm{\rho_t^{\varepsilon,\delta}}_{L^\infty(D)}\norm{(\rho_t^{\varepsilon,\delta})^{\frac{p}{2}}}_{L^2(D)}\norm{\nabla((\rho_t^{\varepsilon,\delta})^{\frac{p}{2}})}_{L^2(D)} \,.
	\end{equation*}
	Applying Cauchy-Schwarz inequality in~\eqref{eq:rho-delta-p-1}, we obtain
	\begin{align*}
	\frac{{\rm d}\norm{\rho_t^{\varepsilon,\delta}}_{L^p(D)}^p}{\dt}
	+\frac{4\nu(p-1)}{p}\norm{\nabla((\rho_t^{\varepsilon,\delta})^{\frac{p}{2}})}_{L^2(D)}^2
	\leq 
	\frac{2\nu(p-1)}{p}\norm{\nabla((\rho_t^{\varepsilon,\delta})^{\frac{p}{2}})}_{L^2(D)}^2+Cp\norm{\rho_t^{\varepsilon,\delta}}_{L^\infty(D)}^2\norm{\rho_t^{\varepsilon,\delta}}_{L^p(D)}^p \,,
	\end{align*}
	where $C$ is independent of $\varepsilon, \delta$, and $p$. Gronwall's inequality then leads to
	\begin{equation*}
	\norm{\rho_t^{\varepsilon,\delta}}_{L^p(D)}
	\leq 
	\norm{\rho^\delta_0}_{L^p(D)}\exp \vpran{C\int_0^t\norm{\rho_s^\varepsilon}_{L^\infty(D)}^2 \ds}
	\qquad
	\text{for any $p \geq 2$\,.}
	\end{equation*}
	Now we send $p\rightarrow\infty$ to get
	\begin{equation*}
	\norm{\rho_t^{\varepsilon,\delta}}_{L^\infty(D)}
	\leq 
	\norm{\rho^\delta_0}_{L^\infty(D)}\exp\left(C\int_0^t\norm{\rho_s^\varepsilon}_{L^\infty(D)}^2 \ds\right) \,.
	\end{equation*}
	Finally, applying Lemma \ref{L2-1} with $f(t)=\norm{\rho_t^{\varepsilon,\delta}}_{L^\infty(D)}^2$ yields that
	\begin{equation*}
	\norm{\rho_t^{\varepsilon,\delta}}_{L^\infty(D)}^2\leq\frac{\norm{\rho_0^\delta}_{L^\infty(D)}^2}{1- 2C\norm{\rho_0^\delta}_{L^\infty(D)}^2t}\leq \frac{\norm{\rho_0}_{L^\infty(D)}^2}{1- 2C\norm{\rho_0}_{L^\infty(D)}^2t}\,.
	\end{equation*}
	Therefore, the uniform existence time $T$ can be chose as
	\begin{align*}
	T = \frac{3}{8 C \norm{\rho_0}^2_{L^\infty(D)}} \,,
	\end{align*}
	such that
	\begin{align*}
	\norm{\rho^{\Eps, \delta}_t}_{L^\infty(D)}
	\leq 
	2 \norm{\rho_0}_{L^\infty(D)}
	\qquad
	\text{on $[0, T)$\,.}
	\end{align*}
	Therefore there exists a $\rho^\Eps_t \in L^\infty\left(0, T; L^\infty(D)\right)$ such that up to a subsequence, 
	\begin{equation}\label{convestar}
	\rho_t^{\varepsilon,\delta}\rightharpoonup^\ast \rho_t^\varepsilon\quad\mbox{in } L^\infty\left(0, T; L^\infty(D)\right)\,.
	\end{equation}
	
	Now we show that $\rho_t^\varepsilon$ is a weak solution to \eqref{RPDE}.  
	Notice that for any $\phi \in C^\infty(\overline D)$ satisfying $\langle \nabla\phi(x),n(x) \rangle=0$ on $\partial D$, we have
	\begin{align} \label{eq:weak-approx}
	\int_{D}\rho_t^{\varepsilon,\delta}(x)\phi(x) \dx
	& = \int_{D}\rho_0^{\delta}(x)\phi(x) \dx
	+\nu\int_0^t\int_{D}  \Delta \phi(x)\rho_s^{\varepsilon,\delta}(x) \dx\ds \nn
	\\
	&\quad \,
	+ \int_0^t \int_{D} \rho_s^{\varepsilon,\delta} (x) F_\varepsilon \ast \rho_s^{\varepsilon,\delta}(x)\cdot\nabla \phi(x) \dx\ds \,.
	\end{align}
	Then by taking a continuous version of the left-hand side and~\eqref{convestar}, we derive that 
	\begin{align*}
	\int_{D} \rho_t^{\varepsilon,\delta}(x)\phi(x) \dx
	&\rightarrow 
	\int_{D}\rho_t^\varepsilon(x)\phi(x) \dx  \,,
	\qquad
	\int_{D}\rho_0^{\delta}(x)\phi(x) \dx
	\rightarrow 
	\int_{D}\rho_0(x)\phi(x) \dx \,,
	\\
	\nu\int_0^t\int_{D} & \Delta \phi(x)\rho_s^{\varepsilon,\delta}(x) \dx\ds \rightarrow 
	\nu\int_0^t\int_{D}  \Delta \phi(x)\rho_s^\varepsilon(x) \dx\ds\,,
	\end{align*}
	as $\delta \to 0$. For the last term on the right-hand side of~\eqref{eq:weak-approx}, we have
	\begin{align*}
	&\left|\int_0^t\int_{D}\rho_s^{\varepsilon,\delta}(x)F_\varepsilon\ast\rho_s^{\varepsilon,\delta}(x)\cdot\nabla \phi(x) \dx\ds
	- \int_0^t\int_{D}\rho_s^\varepsilon(x)F_\varepsilon\ast\rho_s^\varepsilon(x)\cdot\nabla \phi(x) \dx\ds\right| \nn
	\\
	\leq & 
	\left|\int_0^t\int_{D}\rho_s^{\varepsilon,\delta}(x)F_\varepsilon\ast\rho_s^{\varepsilon,\delta}(x)\cdot\nabla \phi(x) \dx\ds
	- \int_0^t\int_{D}\rho_s^{\varepsilon,\delta}(x)F_\varepsilon\ast\rho_s^\varepsilon(x)\cdot\nabla \phi(x) \dx\ds\right| \nn
	\\
	& +\left|\int_0^t\int_{D}\rho_s^{\varepsilon,\delta}(x)F_\varepsilon\ast\rho_s^\varepsilon(x)\cdot\nabla \phi(x) \dx\ds- \int_0^t\int_{D}\rho_s^{\varepsilon}(x)F_\varepsilon\ast\rho_s^\varepsilon(x)\cdot\nabla \phi(x) \dx\ds\right| \nn
	\\
	\leq & 
	C\int_0^t\int_{D} |F_\varepsilon\ast(\rho_s^{\varepsilon,\delta}-\rho_s^{\varepsilon})(x)| \ds\dx
	+ C \left|\int_0^t\int_{D}(\rho_s^{\varepsilon,\delta}-\rho_s^{\varepsilon})\phi(x) \dx\ds \right|\,.
	\end{align*}
	By \eqref{convestar}, one has
	\begin{equation*}
	\left|\int_0^t\int_{D}(\rho_s^{\varepsilon,\delta}-\rho_s^{\varepsilon})\phi(x) \dx\ds \right|\rightarrow 0 
	\quad
	\mbox{ as } \delta\rightarrow0\,.
	\end{equation*}
	Notice that for any $x\in D$, one has
	\begin{equation*}
	|F_\varepsilon\ast(\rho_s^{\varepsilon,\delta}-\rho_s^{\varepsilon})(x)|\leq\int_D F_\varepsilon(x-y)(\rho_s^{\varepsilon,\delta}(y)-\rho_s^{\varepsilon}(y)) \dy\rightarrow 0
	\quad
	\mbox{ as } \delta\rightarrow0\,,
	\end{equation*}
	since $F_\varepsilon\in L^1(D)$. It follows from the dominated convergence theorem that
	\begin{equation*}
	\int_0^t\int_{D} |F_\varepsilon\ast(\rho_s^{\varepsilon,\delta}-\rho_s^{\varepsilon})(x)| \ds\dx \rightarrow 0
	\quad
	\mbox{ as } \delta\rightarrow0\,.
	\end{equation*}
	Therefore,
	\begin{equation*}
	\int_0^t\int_{D}\rho_s^{\varepsilon,\delta}(x)F_\varepsilon\ast\rho_s^{\varepsilon,\delta}(x)\cdot\nabla \phi(x) \dx\ds
	\rightarrow 
	\int_0^t\int_{D}\rho_s^\varepsilon(x)F_\varepsilon\ast\rho_s^\varepsilon(x)\cdot\nabla \phi(x) \dx\ds
	\quad
	\mbox{ as } \delta\rightarrow 0 \,.
	\end{equation*}
	Combining the above limits, we have shown that $\rho^\Eps_t$ is a weak solution to~\eqref{RPDE} in the sense that
	\begin{align*} 
	\int_{D}\rho_t^\varepsilon(x)\phi(x) \dx
	&=\int_{D}\rho_0(x)\phi(x)\dx+\nu\int_0^t\int_{D}  \Delta \phi(x)\rho_s^\varepsilon(x) \dx\ds
	\\
	&\qquad
	+ \int_0^t\int_{D}\rho_s^\varepsilon(x)F_\varepsilon\ast\rho_s^\varepsilon(x)\cdot\nabla \phi(x) \dx\ds \,.
	\end{align*}
	with the bound in~\eqref{bound:rho-Eps}.
\end{proof}


Next, we show that for any weak solution $\rho_t^\varepsilon\in L^\infty\left(0, T; L^\infty(D)\right)$ to \eqref{RPDE}. One can construct a process $Y_t^\varepsilon$ with $\rho_t^\varepsilon$ being its density function, which satisfies regularized self-consistent SDE
\begin{align} \label{RSDE}
\begin{cases}
{Y}_t^{\varepsilon}
={Y}_0 + \int_{0}^t\int_{D}F_\varepsilon({Y}_s^{\varepsilon}-y)\rho_s^\varepsilon(y) \dy\ds
+\sqrt{2\nu}{B}_t - \tilde R_t^{\varepsilon}, & t>0 \,,
\\[2pt]
\tilde R_t^{\varepsilon}
= \int_0^t n(Y_s^{\varepsilon}) {\, \rm d} |\tilde R^{\varepsilon}|_s \,, 
\quad
|\tilde R^{\varepsilon}|_t
=\int_0^t\textbf{1}_{\partial D}(Y_s^{\varepsilon}) {\, \rm d}|\tilde R^{\varepsilon}|_s \,,
\end{cases}
\end{align}
where we require $Y_t^\varepsilon$ possessing a density function $\rho_t^\varepsilon(x)$. The result states as follows


\begin{proposition}\label{relation1}
	If $\rho_t^\varepsilon\in L^\infty\left(0, T; L^\infty(D)\right)$ is a weak solution to \eqref{RPDE} with the initial data $\rho_0\in L^\infty(D)$, then given any random variable $Y_0$ with the density $\rho_0$, there is a process $Y_t^\varepsilon$ with the marginal density $\rho_t^\varepsilon$ satisfying \eqref{RSDE} with the initial data $Y_0$.
\end{proposition}
\begin{proof}
	Consider the following linear PDE
	\begin{eqnarray}\label{lineareq}
	\left\{\begin{array}{l}
	\partial_t\rho_t^\varepsilon=\nu\triangle \rho_t^\varepsilon-\nabla \cdot[\rho_t^\varepsilon V_\varepsilon [g_t]]\,, \quad x\in D\,, \quad t>0\,,\\
	\rho_t(x)|_{t=0}=\rho_0(x)\,,~x\in D \mbox{ and } \langle \nu\nabla\rho_t^\varepsilon-\rho_t V_\varepsilon [g_t]\,,n\rangle=0\mbox{ on } \partial D\,,
	\end{array}\right.
	\end{eqnarray}
	where $V _\varepsilon[g_t]=\int_{D}F_\varepsilon(x-y)g_t(y)dy$, $g_t(x)\in L^\infty\left(0, T; L^\infty(D)\right)$ is a given function. According to Lemma \ref{L1} $(i)$, the term $V_\varepsilon[g_t]$ is bounded, thus 
	the linear PDE \eqref{lineareq} has a unique weak solution in $L^\infty(0,T;L^1(D))$. 
	
	Next let $V_\varepsilon[\rho_t^\varepsilon]=\int_{D}F_\varepsilon(x-y)\rho_t^\varepsilon(y) dy$, where $\rho_t^\varepsilon\in L^\infty\left(0, T; L^\infty(D)\right)$ is a weak solution to \eqref{RPDE}.   According to \cite[Theorem 1.1]{sznitman1984nonlinear}, the following SDE
	\begin{eqnarray}\label{RSDElinear}
	\left\{\begin{array}{l}
	{Y}_t^{\varepsilon}={Y}_0+\int_{0}^tV_\varepsilon[\rho_t^\varepsilon]({Y}_s^\varepsilon) \ds+\sqrt{2\nu}{B}_t-\tilde R_t^{\varepsilon}\,, \quad t>0\,,
	\\[2pt]
	\tilde R_t^{\varepsilon}=\int_0^tn(Y_s^{\varepsilon}) {\, \rm d}|\tilde R^{\varepsilon}|_s\,,\quad |\tilde R^{\varepsilon}|_t=\int_0^t\textbf{1}_{\partial D}(Y_s^{\varepsilon}) {\, \rm d}|\tilde R^{\varepsilon}|_s\,,\\
	\end{array}\right.
	\end{eqnarray}
	has a strong solution $Y_t^\varepsilon$ with density $\bar\rho_t^\varepsilon$ since $V_\varepsilon[\rho_t^\varepsilon]$ is smooth. 
	For $\phi\in C^\infty(\overline D)$ satisfying $\langle \nabla\phi(x),n(x) \rangle=0$ for $x\in\partial D$, one applies It\^{o}'s formula:
	\begin{align}\label{ito}
	\phi(Y_t^\varepsilon)&=\phi(Y_0)+\int_0^t\langle \nabla\phi(Y_s^\varepsilon), {\, \rm d} Y_s^\varepsilon\rangle+\frac{1}{2}\int_0^t (d Y_s^\varepsilon)^T\nabla^2 \phi(Y_s^\varepsilon) {\, \rm d} Y_s^\varepsilon \notag\\
	&=\phi(Y_0)+\int_0^t\langle \nabla\phi(Y_s^\varepsilon),V_\varepsilon[\rho_t^\varepsilon]({Y}_s^\varepsilon)\rangle \ds+\sqrt{2\nu}\int_0^t\langle \nabla\phi(Y_s^\varepsilon), {\, \rm d}B_s\rangle \notag\\
	&\quad -\int_0^t\langle \nabla\phi(Y_s^\varepsilon), {\, \rm d}\tilde R_s^\varepsilon\rangle +\nu\int_0^t\Delta\phi(Y_s^\varepsilon) \ds\,.
	\end{align}
	Using the boundary condition, one has
	\begin{equation*}
	\int_0^t\langle \nabla\phi(Y_s^\varepsilon), {\, \rm d}\tilde R_s^\varepsilon\rangle =\int_0^t\langle \nabla\phi(Y_s^\varepsilon),n(Y_s^\varepsilon)\rangle {\, \rm d}|\tilde R^\varepsilon|_s=0\,.
	\end{equation*}
	Taking expectation on \eqref{ito}, it follows from above that
	\begin{align*}
	\int_{D}\bar\rho_t^\varepsilon(x)\phi(x) \dx
	&= \int_{D}\rho_0(x)\phi(x) \dx
	+\nu\int_0^t\int_{D}  \Delta \phi(x)\bar\rho_s^\varepsilon(x) \dx\ds
	+ \int_0^t\int_{D}\bar\rho_s^\varepsilon(x)V_\varepsilon[\rho_t^\varepsilon]\cdot\nabla \phi(x) \dx\ds\,.
	\end{align*}
	This implies that $\bar\rho_t^\varepsilon$ is a weak solution to \eqref{lineareq} associated to $V_\varepsilon[\rho_t^\varepsilon]$. Notice that $\rho_t^\varepsilon$ is also a weak solution to \eqref{lineareq} associated to $V_\varepsilon[\rho_t^\varepsilon]$. By the uniqueness of weak solution to \eqref{lineareq}, we obtain $\rho_t^\varepsilon=\bar\rho_t^\varepsilon$.  Hence we have proved that there exists a process $Y_t^\varepsilon$ with density $\rho_t^\varepsilon$， which is a strong solution to the regularized mean-field dynamics \eqref{RSDE}.
	%
	%
\end{proof}

\subsection{The original system}
Now we show the proof of Theorem~\ref{mainthm}, which asserts the existence of solutions to the SDE~\eqref{SDE'} and the KS equation~\eqref{PDE}.

\begin{proof}[Proof of Theorem \ref{mainthm}] 
	Note that the existence of weak solutions to the KS equation \eqref{PDE} follows from applying the It\^{o}'s formula on SDE~\eqref{SDE'}. Thus we focus on proving the existence of a strong solution to the SDE~\eqref{SDE'}, which is obtained by first constructing approximate solutions to the regularized SDE~\eqref{RSDE} and then passing to the limit. 
	
	
	To this end, we let $\varepsilon>\varepsilon'>0$ be arbitrary and $F_\Eps, F_{\Eps'}$ be the corresponding regularized forces. 
	Denote  $\rho_t^\varepsilon$, $\rho_t^{\varepsilon'}$ are two weak solutions to \eqref{RPDE} with the same initial condition $\rho_0\in L^\infty(D)$. Then there exists $T > 0$ independent of $\Eps, \Eps'$ such that $\norm{\rho_t^\varepsilon}_{L^\infty(D)}$ and $\norm{\rho_t^{\varepsilon'}}_{L^\infty(D)}$ are uniformly bounded in $\varepsilon$ and $\varepsilon'$  on $[0, T]$.  According to Proposition \ref{relation1}, we can construct two processes ${Y}^\varepsilon_t$, ${Y}^{\varepsilon'}_t$ satisfying \eqref{RSDE} with the same initial data ${Y}_0$ and Brownian motion $B_t$,  which have density functions $\rho_t^\varepsilon$ and $\rho_t^{\varepsilon'}$ respectively. 
	We shall show that there exist constants $C$, $C_T$, and $\varepsilon_0$ depending only on $D$, $T$ and $\norm{\rho_0}_\infty$ such that if $0 < \Eps' < \varepsilon< \varepsilon_0$, then
	\begin{align}\label{cauchy}
	\sup\limits_{t\in[0,T]}\mathcal {W}_\infty(\rho_t^\varepsilon,\rho_t^{\varepsilon'})
	\leq 
	\sup\limits_{t\in[0,T]}\PP\mbox{-ess sup }|{Y}^\varepsilon_t-{Y}^{\varepsilon'}_t| \leq C_T\omega(\varepsilon)^{\frac{1}{2}e^{-CT}}\,.
	\end{align}
	Indeed, by applying It\^{o}'s formula to $Y^\Eps_t - Y^{\Eps'}_t$, we have
	\begin{align*} 
	|Y_t^\varepsilon-Y_t^{\varepsilon'}|^2
	&=2\int_0^t\la Y_s^\varepsilon-Y_s^{\varepsilon'},F_\varepsilon\ast\rho_s^\varepsilon(Y_s^\varepsilon)-F_{\varepsilon'}\ast\rho_s^{\varepsilon'}(Y_s^{\varepsilon'})\ra \ds  \nn
	\\ 
	&\quad 
	-2\int_0^t\la Y_s^\varepsilon-Y_s^{\varepsilon'},n(Y_s^{\varepsilon})\ra {\rm d}|\tilde R^\varepsilon|_s-2\int_0^t
	\la Y_s^{\varepsilon'}-Y_s^{\varepsilon},n(Y_s^{\varepsilon'})\ra {\rm d} |\tilde R^{\varepsilon'}|_s \,.
	\end{align*}
	By the convexity of the domain $D$, one has
	\begin{align*}
	(x-y)\cdot n(x)\geq 0\quad \mbox{for any } x\in\partial D\mbox{ and } y\in \overline D\,.
	\end{align*}
	Consequently, 
	\begin{align}\label{goodsign}
	\la Y_s^\varepsilon-Y_s^{\varepsilon'},n(Y_s^{\varepsilon})\ra 
	\geq 0\,,
	\quad 
	{\rm d}|\tilde R^\varepsilon|_s \mbox{ almost surely}\,,
	\end{align}
	and
	\begin{align}\label{goodsign1}
	\la Y_s^{\varepsilon'}-Y_s^{\varepsilon},n(Y_s^{\varepsilon'})\ra  
	\geq 0,
	\quad 
	{\rm d}|\tilde R^{\varepsilon'}|_s \mbox{ almost surely} \,.
	\end{align}
	Therefore, by Lemma \ref{QLe}, we have
	\begin{align*} 
	|Y_t^\varepsilon-Y_t^{\varepsilon'}|^2
	&\leq 
	2\int_0^t\langle Y_s^\varepsilon-Y_s^{\varepsilon'},F_\varepsilon\ast\rho_s^\varepsilon(Y_s^\varepsilon)-F_{\varepsilon'}\ast\rho_s^{\varepsilon'}(Y_s^{\varepsilon'})\rangle \ds \nn
	\\
	&\leq 2\int_0^t | Y_s^\varepsilon-Y_s^{\varepsilon'}| \left|F_\varepsilon\ast\rho_s^\varepsilon(Y_s^\varepsilon)-F_{\varepsilon'}\ast\rho_s^{\varepsilon'}(Y_s^{\varepsilon'})\right| \ds\notag \\
	&\leq 2 \int_0^t | Y_s^\varepsilon-Y_s^{\varepsilon'}| \left( C\omega(\varepsilon)+C\omega(\PP\mbox{-ess sup }| Y_s^\varepsilon-Y_s^{\varepsilon'}| )\right) \ds\notag \\
	&\leq C\omega(\varepsilon) T+C\int_0^t | Y_s^\varepsilon-Y_s^{\varepsilon'}| \omega(\PP\mbox{-ess sup }| Y_s^\varepsilon-Y_s^{\varepsilon'}| ) \ds\,,
	\end{align*}
	where $C$ only depends on $D$, $\norm{\rho_t^\varepsilon}_\infty$,  and $\norm{\rho_t^{\varepsilon'}}_\infty$.
	Hence, by Lemma \ref{L2}, there exist constants $C_T$, $\varepsilon_0$ depending only on $C$ and $T$ such that if $\varepsilon< \varepsilon_0$, then $\omega(\varepsilon)<\varLambda_0$ and the following holds
	\begin{align} \label{part2}
	\sup\limits_{t\in[0,T]}\PP\mbox{-ess sup }|{Y}^\varepsilon_t-{Y}^{\varepsilon'}_t| 
	\leq 
	C_T\omega(\varepsilon)^{\frac{1}{2}e^{-CT}},
	\end{align}
	Therefore \eqref{cauchy} holds. Hence
	there exists a limiting stochastic process ${Y}_t \in C([0,T],D)$ such that
	\begin{align}\label{caucy1}
	\sup\limits_{t\in[0,T]}\PP\mbox{-ess sup }|{Y}^\varepsilon_t-Y_t| \rightarrow0
	\qquad
	\mbox{ as }\varepsilon\rightarrow0 \,.
	\end{align}
	Assume that $Y_t$ has the marginal density $\rho_t$. Then by the bound
	\begin{align*}
	\sup\limits_{t\in[0,T]} \mathcal{W}_\infty(\rho_t^\varepsilon,\rho_t)\leq \sup\limits_{t\in[0,T]}\PP\mbox{-ess sup }|{Y}^\varepsilon_t-{Y}_t| \,,
	\end{align*}
	we deduce that 
	\begin{align*}
	\sup\limits_{t\in[0,T]} \mathcal{W}_\infty(\rho_t^\varepsilon,\rho_t)\rightarrow0 
	\qquad
	\mbox{ as }\varepsilon\rightarrow0\,,
	\end{align*}
	and $\rho_t\in L^\infty\left(0, T; L^\infty(D)\right)$.
	
	We are left to show that there exists $\widetilde R_t$ such that $(Y_t, \widetilde R_t)$ is a strong solution to the SDE~\eqref{SDE'}. 
	By Lemma \ref{L2} and \eqref{caucy1},
	\begin{align} \label{convg:Y-Eps}
	\left|\int_0^t(F_\varepsilon\ast \rho_s^\varepsilon(Y_s^\varepsilon)-F\ast \rho_s(Y_s)) \ds\right|
	\leq 
	C\int_{0}^{t}\omega(\PP\mbox{-ess sup }|{Y}^{\varepsilon}_s-{Y}_s|) \ds+ C\omega(\varepsilon) t\rightarrow0
	\quad 
	\mbox{ uniformly in $t$}
	\end{align}
	as $\Eps \to 0$.
	We define the boundary process $\tilde R_t$ as
	\begin{equation*}
	\tilde R_t
	:= - \vpran{Y_t-Y_0-\int_0^tF\ast\rho_s(Y_s)\ds-\sqrt{2\nu} B_t}\, ,
	\end{equation*}
	such that $(Y_t, \widetilde R_t)$ is the strong solution to the SDE~\eqref{SDE'}. The properties of $R_t$ in~\eqref{SDE'} are verified in the same way as in \cite[Step \textbf{B} on page 13]{choi2016propagation}, whose details are omitted here. 
	%
	%
	%
\end{proof}

Similar to Proposition \ref{relation1}, we show that any weak solution $\rho_t\in L^\infty\left(0, T; L^\infty(D)\right)$ to \eqref{PDE} can be represented as the family of time marginal density of a solution to \eqref{SDE'}.
\begin{proposition}\label{relation}
	If $\rho_t\in L^\infty\left(0, T; L^\infty(D)\right)$ is solution to \eqref{PDE} with the initial data $\rho_0\in L^\infty(D)$, then given any random variable $Y_0$ with the density $\rho_0$, there is a process $Y_t$ with the marginal density $\rho_t$ satisfying \eqref{SDE'} with the initial data $Y_0$.
\end{proposition}
\begin{proof}
	Proposition \ref{relation1} is a regularized version of this proposition, so the proof is almost the same. To this end, we change $V_\varepsilon[g_t]$ into $V [g_t]=\int_{D}F(x-y)g_t(y)\dy$. The details are omitted here.
\end{proof}

Using Proposition \ref{relation}, we prove the stability estimates  in Theorem \ref{mainthm1}.
\begin{proof}[Proof of Theorem \ref{mainthm1}]
	
	
	First we prove stability of the nonlinear SDE \eqref{SDE'}.
	Using similar arguments as in the proof of Theorem \ref{mainthm}, we apply It\^{o}'s formula and get
	\begin{align}\label{ito2}
	|Y_t^1-Y_t^{2}|^2&=|Y_0^1-Y_0^{2}|^2+2\int_0^t\la Y_s^1-Y_s^{2},F\ast\rho_s^1(Y_s^1)-F\ast\rho_s^{2}(Y_s^{2})\ra \ds\notag
	\\
	&\quad
	-2\int_0^t\la Y_s^1-Y_s^{2},n(Y_s^{1})\ra {\, \rm d}|\tilde R^1|_s-2\int_0^t
	\la Y_s^{2}-Y_s^{1},n(Y_s^{2})\ra {\, \rm d}|\tilde R^{2}|_s\notag
	\\
	&\leq |Y_0^1-Y_0^{2}|^2+2\int_0^t\la Y_s^1-Y_s^{2},F\ast\rho_s^1(Y_s^1)-F\ast\rho_s^{2}(Y_s^{2})\ra \ds\,.
	\end{align}
	Then it follows from Lemma \ref{QLe} that there exists a constant $C$ depending only on $\|\rho^1\|_{L^\infty(0,T; {L^\infty(D)})}$ and $\|\rho^2\|_{L^\infty(0,T; {L^\infty (D)})}$
	such that
	\begin{align}\label{diffpath}
	|Y_t^1-Y_t^{2}|^2\leq |Y_0^1-Y_0^{2}|^2+C\int_{0}^{t}| Y_s^1-Y_s^{2}|\omega(\PP\mbox{-ess sup }| Y_s^1-Y_s^2| ) \ds\,.
	\end{align}
	The bound of $|Y_t^1-Y_t^{2}|$ is derived by considering three possible cases:
	
	\textit{Case 1:} There exists a small enough constant $C_0(T) $ such that $\PP\mbox{-ess sup }|{Y}^1_0-{Y}^2_0|\leq C_0(T)$. Here $C_0(T)$ is determined by Lemma  \ref{L2}. Actually, we can treat $\PP\mbox{-ess sup }|Y_t^1-Y_t^{2}|$ as $f_\varepsilon(t)$ in Lemma \ref{L2} and $|Y_0^1-Y_0^{2}|^2$ as $C\omega(\varepsilon)T$. Suppose that $\PP\mbox{-ess sup }|{Y}^1_0-{Y}^2_0|<C_0(T)<1$ sufficiently small. Then there exists two constants $C_T$ and $0<\varLambda _0<\frac{1}{2}$ depending only on $C$ and $T$ such that if $|Y_0^1-Y_0^{2}|^2<C_0(T)^2<\varLambda_0$, then
	\begin{eqnarray}\label{87}
	\PP\mbox{-ess sup }|{Y}^1_t-{Y}^2_t|\leq C_T \{\PP\mbox{-ess sup }|{Y}^1_0-{Y}^2_0|\}^{\exp(-CT)}<1 \quad \mbox ~{for~ all~} t\in [0, T]\,.
	\end{eqnarray}
	
	\textit{Case 2:} If $\PP\mbox{-ess sup }|{Y}^1_0-{Y}^2_0|\geq C_0(T)$ and  $\PP\mbox{-ess sup }|{Y}^1_t-{Y}^2_t|\geq C_0(T)$ for any $t\in [0,T]$, then
	by the definition of $\omega(r)$,  one  obtains that there exists a constant $C_1$ depending only on $C_0(T)$ such that 
	\begin{equation*}
	\omega(\PP\mbox{-ess sup }|{Y}^1_t-{Y}^2_t|)\leq C_1 \PP\mbox{-ess sup }|{Y}^1_t-{Y}^2_t|\quad \mbox ~{for~ all~} t\in [0, T]\,.
	\end{equation*}
	Thus it follows from \eqref{diffpath} that
	\begin{eqnarray}\label{diffpathre}
	|{Y}^1_t-{Y}^2_t|^2 \leq |{Y}^1_0-{Y}^2_0|^2+C_1C\int_0^t|{Y}^1_s-{Y}^2_s| \PP\mbox{-ess sup }|{Y}^1_s-{Y}^2_s|\ds\,.
	\end{eqnarray}
	Applying the Gronwall inequality, one has for any $t\in [0, T]$,
	\begin{align}\label{stabilityest3}
	\PP\mbox{-ess sup }|{Y}^1_t-{Y}^2_t|&\leq \PP\mbox{-ess sup }|{Y}^1_0-{Y}^2_0|\left(1+CC_1Te^{CC_1T}\right)^\frac{1}{2}\,.
	\end{align}
	
	\textit{Case 3:} If $\PP\mbox{-ess sup }|{Y}^1_0-{Y}^2_0|\geq C_0(T)$ and there exists a $t_0\in (0, T)$ such that $\PP\mbox{-ess sup }|{Y}^1_t-{Y}^2_t|\geq C_0(T)$ for $t\in [0, t_0)$ and $\PP\mbox{-ess sup }|{Y}^1_{t_0}-{Y}^2_{t_0}|= C_0(T)$.
	In this case, the estimate for the interval $t\in [0, t_0)$ is reduced to the \textit{Case $2$} and one has
	\begin{equation}\label{stabilityest4}
	\PP\mbox{-ess sup }|{Y}^1_t-{Y}^2_t|\leq \PP\mbox{-ess sup }|{Y}^1_0-{Y}^2_0|\left(1+CC_1Te^{CC_1T}\right)^\frac{1}{2}.
	\end{equation}
	Over the interval $[t_0,T]$, choosing $t_0$ as a new initial time and repeating the proof in Case 1 gives 
	\begin{eqnarray}\label{stabilityest5}
	\PP\mbox{-ess sup }|{Y}^1_t-{Y}^2_t|\leq C_T \{\PP\mbox{-ess sup }|{Y}^1_{t_0}-{Y}^2_{t_0}|\}^{\exp(-C(T-t_0))} \quad \mbox ~{for~ all~} t\in [t_0, T]\,.
	\end{eqnarray}
	By \eqref{stabilityest4} and the continuity of $\PP\mbox{-ess sup }|{Y}^1_t-{Y}^2_t|$, one has 
	\begin{equation}\label{93}
	\PP\mbox{-ess sup }|{Y}^1_{t_0}-{Y}^2_{t_0}|\leq \PP\mbox{-ess sup }|{Y}^1_0-{Y}^2_0|\left(1+CC_1Te^{CC_1T}\right)^\frac{1}{2}\,.
	\end{equation}
	Combining \eqref{stabilityest4}, \eqref{stabilityest5} and \eqref{93},  we obtain
	\begin{align}\label{stabilityest6}
	\PP\mbox{-ess sup }|{Y}^1_t-{Y}^2_t|
	\leq 
	C_2 \max \left\{\PP\mbox{-ess sup }|{Y}^1_0-{Y}^2_0|, \,\, \{\PP\mbox{-ess sup }|{Y}^1_0-{Y}^2_0|\}^{\exp(-C_3T)}\right\} \,,
	\end{align}
	for any $t\in [0, T]$. Combining \eqref{87}, \eqref{stabilityest3} and  \eqref{stabilityest6} finishes the proof of \eqref{stab1}.

	
	Next we prove the stability of the KS equation \eqref{PDE}.
	Let $\rho^1, \rho^2 \in L^\infty\left(0, T; L^\infty(D)\right)$ be two weak solutions to \eqref{PDE} with the initial conditions $\rho_0^1, \rho_0^2 \in L^\infty(D)$ respectively. Take two initial random variables ${Y}_0^1$ and ${Y}_0^2$ such that
	\begin{equation}\label{initialcondiconst}
	\mathcal{W}_\infty (\rho_0^1,~\rho_0^2)=\PP\mbox{-ess sup }|{Y}^1_0-{Y}^2_0| \,.
	\end{equation}
	By Proposition~\ref{relation}, we build two self-consistent SDEs such that
	\begin{align*}
	\begin{cases}
	{Y}_t^i={Y}_0^i+\int_{0}^t\int_{D}F({Y}_s^i-y)\rho_s^i(y) \dy\ds+\sqrt{2\nu}{B}_t-\tilde R_t^i\,, \quad t>0,~i=1,2,
	\\[3pt]
	\tilde R_t^i=\int_0^tn(Y_s^i) {\, \rm d}|\tilde R^i|_s\,,
	\quad 
	|\tilde R^i|_t=\int_0^t\textbf{1}_{\partial D}(Y_s^i) {\, \rm d}|\tilde R^i|_s \,,
	\end{cases}
	\end{align*}
	which have unique strong solutions $({Y}^i_t)_{t\geq0}$ with $(\rho^i_t)_{t\geq0}$ being the marginal density of $({Y}^i_t)_{t\geq0}$. 
	Hence it follows from \eqref{stab1} and \eqref{initialcondiconst} that
	\begin{align*}
	\sup\limits_{t\in{[0,T]}}\mathcal{W}_\infty (\rho_t^1,\rho_t^2)&\leq \sup\limits_{t\in{[0,T]}}\PP\mbox{-ess sup }|{Y}^1_t-{Y}^2_t|\notag\\
	&\leq C_1\max \left\{\PP\mbox{-ess sup }|{Y}^1_0-{Y}^2_0|, \{\PP\mbox{-ess sup }|{Y}^1_0-{Y}^2_0|\}^{\exp(-C_2T)}\right\}\notag \\
	&= C_1 \max \left\{\mathcal{W}_\infty (\rho_0^1,~\rho_0^2), \{\mathcal{W}_\infty (\rho_0^1,~\rho_0^2)\}^{\exp(-C_2T)}\right\},
	\end{align*}
	which leads to the desired stability estimate \eqref{stab2}.
\end{proof}
\section{Propagation of chaos} \label{sec:prop-chaos}
In this section, we show an estimate of the collision time between particles and prove the propagation of chaos result stated in Theorem \ref{mainthm2}.

\subsection{Collision time between particles} 
It is known that in $\R^2$, if $8\pi\nu<\norm{\rho_0}_1=1$, then solutions to the Keller-Segel equation \eqref{PDE} concentrate within a finite time \cite{dolbeault2004optimal}. In this part, we show a similar bound for the expectation of the collision time for the interacting particle system \eqref{particle}. 
\begin{thm}\label{collison}
	Assume $d=2$ and $\{X_0^i\}_{i=1}^N$ be $N$ i.i.d. random variables  with a common density $\rho_0$. Let $\{(X_t^i,R_t^i)\}_{i=1}^N$ be weak solutions to \eqref{particle} with the initial data $\{X_0^i\}_{i=1}^N$. For any fixed $T>0$ and $8\pi\nu<1$, we define
	\begin{equation*}
	A(t):=\inf\limits_{s\in[0,t]}\min\limits_{i\neq j}|X_s^i-X_s^j|\,,
	\qquad
	\tau_\varepsilon
	=\begin{cases}
	\begin{aligned}
	&0  && \text{ if } \varepsilon\geq A(0) \,,\\
	&\sup\{t\wedge T: A(t)\geq \varepsilon\} && \text{ if } \varepsilon< A(0) \,,
	\end{aligned}
	\end{cases}
	\end{equation*}
	where $t \wedge T = \min\{t, T\}$. Let $\tau=\lim\limits_{\varepsilon\rightarrow 0}\tau_\varepsilon$. Then the expectation of the collision time satisfies
	\begin{equation*}
	\EE(\tau)\leq \frac{2\pi  \EE\left[|X_0^1-X_0|^2\right]}{1 -8\pi\nu} \,,
	\qquad   X_0:=\frac{1}{N}\sum\limits_{i=1}^NX_0^i\,.
	\end{equation*}
\end{thm}
\begin{proof} Since system \eqref{particle} has a weak solution until the collision time $\tau$ 
	and $F_\varepsilon(x)=F(x)$ for any $|x|>\varepsilon$, we know that $X_t^i\equiv X_t^{i,\varepsilon}$ and $F_\varepsilon(x)=F(x)$ for $t\in [0,\tau_\varepsilon]$, where $\{X_t^{i,\varepsilon}\}_{i=1}^N$ is a global solution to the regularized particle system \eqref{Rparticle}. Let us denote $$X_0:=\frac{1}{N}\sum\limits_{i=1}^NX_0^i\,.$$
	By the It\^{o}'s formula and \eqref{Rparticle},  for $t\in[0,\tau_\varepsilon]$, one has
	\begin{align}\label{1collap}
	|{X}^{i,\varepsilon}_t-X_0|^2
	&= |{X}_0^{i}-X_0|^2+ 2\int_0^t\left \la {X}^{i,\varepsilon}_s-X_0, \,\,\frac{1}{N-1}\sum\limits_{j\neq i}^N F({X}^{i,\varepsilon}_s-{X}^{j,\varepsilon}_s)\right\ra \ds   
	\notag 
	\\
	&\quad + 2\int_0^t \la {X}^{i,\varepsilon}_s-X_0, \,\, \sqrt{2\nu} {\, \rm d}{B}_s^{i,\varepsilon}\ra
	- 2\int_0^t \la {X}^{i,\varepsilon}_s-X_0, \,\, {\, \rm d}R_s^{i,\varepsilon}\ra+4\nu t \notag 
	\\
	& \hspace{-0.3cm}
	\leq 	
	|{X}_0^{i}-X_0|^2+ 2\int_0^t\left \la {X}^{i,\varepsilon}_s-X_0, \,\,\frac{1}{N-1}\sum\limits_{j\neq i}^N F({X}^{i,\varepsilon}_s-{X}^{j,\varepsilon}_s)\right\ra \ds   
	\notag 
	\\
	&\quad + 2\int_0^t \la {X}^{i,\varepsilon}_s-X_0, \,\, \sqrt{2\nu} {\, \rm d}{B}_s^{i,\varepsilon}\ra+4\nu t\, ,
	\end{align}
	where we have used the convexity of the domain $D$ such that
	\begin{equation*}
	\int_0^t \la {X}^{i,\varepsilon}_s-X_0, {\, \rm d}R_s^i\ra
	= \int_0^t \la {X}^{i,\varepsilon}_s-X_0,  \, \, n({X}^{i,\varepsilon}_t) \ra {\, \rm d}|R^i|_s \geq 0\,.
	\end{equation*}	
	Then summing \eqref{1collap} for all $i$'s, one has
	\begin{align}\label{2collap}\nonumber
	\sum\limits_{i=1}^N|{X}^{i,\varepsilon}_t-X_0|^2
	&\leq  \sum\limits_{i=1}^N|{X}_0^{i}-X_0|^2+
	\frac{2}{N-1}\sum\limits_{ \scriptstyle  i, j =1 \atop \scriptstyle i \neq j}^N
	\int_0^t\la {X}^{i,\varepsilon}_s -X_0, \,\, F({X}^{i,\varepsilon}_s-{X}^{j,\varepsilon}_s)\ra \ds \\
	&\quad+2\sqrt{2\nu}\sum\limits_{i=1}^N\int_0^t\la {X}^{i,\varepsilon}_s-X_0, \,\, {\, \rm d}{B}_s^{i,\varepsilon}\ra+4\nu Nt\,.
	\end{align}
	Note that $$\int^{\tau_\Eps}_0\mathbb{E}\big[|{X}^{i,\varepsilon}_t-X_0|^2\big] \dt< \infty\,,$$ 
	since $D$ is bounded, which implies that \cite[pp.28, Theorem 1]{gihman1979stochastic}
	\begin{eqnarray}\label{23collap}
	\mathbb{E}\left[\int_0^{t}\la {X}^{i,\varepsilon}_t-X_0, \,\,{\rm d} {B}_t^{i,\varepsilon}\ra \right]=0\,.
	\end{eqnarray}	
	Also note that
	\begin{align*}
	\frac{2}{N-1}\sum\limits_{ \scriptstyle  i, j =1 \atop \scriptstyle i \neq j}^N \la {X}^{i,\varepsilon}_s - X_0, F({X}^{i,\varepsilon}_s-{X}^{j,\varepsilon}_s)\ra
	= \frac{1}{N-1}\sum\limits_{ \scriptstyle  i, j 
		=1 \atop \scriptstyle i \neq j}^N \la {X}^{i,\varepsilon}_s-{X}^{j,\varepsilon}_s, \,\,  F({X}^{i,\varepsilon}_s-{X}^{j,\varepsilon}_s)\ra 
	=-\frac{N}{2\pi}\,.
	\end{align*}
	Then we take expectation on \eqref{2collap} and choose $t=\tau_\Eps\varepsilon$. By the exchangeability of $\{({X}^{i,\varepsilon}_t)_{t\geq0}\}_{i=1}^N$, one has
	\begin{eqnarray}\label{collap}
	0\leq \mathbb{E}\big[|{X}^{i,\varepsilon}_{\tau_\varepsilon}-X_0|^2\big]
	\leq  \EE\left[|{X}_0^{i}-X_0|^2\right]+
	\vpran{4\nu-\frac{1}{2\pi}}\mathbb{E}[\tau_\varepsilon]=\EE\left[|{X}_0^{1}-X_0|^2\right]+
	\vpran{4\nu-\frac{1}{2\pi}}\mathbb{E}[\tau_\varepsilon]\,.
	\end{eqnarray}
	By the positivity of left hand and $8\pi\nu<1$, we obtain
	\begin{equation*}
	\mathbb{E}[\tau_\varepsilon]\leq \frac{2\pi \EE\left[|{X}_0^{1}-X_0|^2\right]}{1 -8\pi\nu}\,.
	\end{equation*}
	Hence, the desired bound for $\tau$ holds by the monotone convergence theorem.
\end{proof}

\subsection{ Propagation of chaos and the proof of Theorem \ref{mainthm2}}

In order to prove Theorem \ref{mainthm2}, we first show the following lemma:
\begin{lem}\label{lmconsist}
	Let $\{X_t^{i,\varepsilon}\}_{i=1}^N$ be $N$ exchangeable random variables on $D$ satisfying \eqref{Rparticle}. Let $\{Y_t^{i,\varepsilon}\}_{i=1}^N$ be $N$ i.i.d. random variables on $D$ satisfying \eqref{RSDE} with the common density $\rho_t^\varepsilon\in L^\infty\left(0,T;L^\infty(D)\right)$. Then there exists a constant $C>0$ depending only on $T$, $D$ and $\norm{\rho^\varepsilon}_{L^\infty\left(0,T;L^\infty(D)\right)}$ such that
	\begin{align}\label{hn}
	\sup\limits_{i=1,\cdots,N}\left|\frac{1}{N-1}\sum\limits_{j\neq i}^NF_\varepsilon(X_t^{i,\varepsilon}-X_t^{j,\varepsilon})-\int_DF_\varepsilon(Y_t^{i,\varepsilon}-y)\rho_t^\varepsilon(y)\dy\right|\notag
	\leq  C\varepsilon^{-d}\sup\limits_{i=1,\cdots,N}|X_t^{i,\varepsilon}-Y_t^{i,\varepsilon}|+\mathcal{H}_N\,,
	\end{align}
	where $\mathcal{H}_N$ is given by
	\begin{equation*}
	\mathcal{H}_N:=\sup\limits_{i=1,\cdots,N}\left|\frac{1}{N-1}\sum\limits_{j\neq i}^NF_\varepsilon(Y_t^{i,\varepsilon}-Y_t^{j,\varepsilon})-\int_DF_\varepsilon(Y_t^{i,\varepsilon}-y)\rho_t^\varepsilon(y)\dy\right|\,,
	\end{equation*}
	and it satisfies
	\begin{equation*}
	\EE[|\mathcal{H}_N|^2]\leq  \frac{C\varepsilon^{-2(d-1)}}{N-1}\,.
	\end{equation*}
\end{lem}
\begin{proof}
	For $i=1$, we split the error
	\begin{align*}
	&\left|\frac{1}{N-1}\sum\limits_{j=2}^NF_\varepsilon(X_t^{1,\varepsilon}-X_t^{j,\varepsilon})-\int_DF_\varepsilon(Y_t^{1,\varepsilon}-y)\rho_t^\varepsilon(y)\dy\right|\notag\\
	\leq& \left|\frac{1}{N-1}\sum\limits_{j=2}^NF_\varepsilon(X_t^{1,\varepsilon}-X_t^{j,\varepsilon})-\frac{1}{N-1}\sum\limits_{j= 2}^NF_\varepsilon(Y_t^{1,\varepsilon}-Y_t^{j,\varepsilon})\right|\notag \\
	&+\left|\frac{1}{N-1}\sum\limits_{j=2}^NF_\varepsilon(Y_t^{1,\varepsilon}-Y_t^{j,\varepsilon})-\int_DF_\varepsilon(Y_t^{1,\varepsilon}-y)\rho_t^\varepsilon(y)\dy\right|
	=:|I_1|+|I_2|\,.
	\end{align*}	
	By Lemma \ref{lmkenerl} $(b)$ one has
	\begin{equation*}
	|I_1|\leq \frac{1}{N-1}\sum\limits_{j=2}^N C\varepsilon^{-d}(|X_t^{1,\varepsilon}-Y_t^{1,\varepsilon}|+|X_t^{j,\varepsilon}-Y_t^{j,\varepsilon}|)=C\frac{\varepsilon^{-d}}{N-1}\sum\limits_{j=1}^N|X_t^{j,\varepsilon}-Y_t^{j,\varepsilon}| \,.
	\end{equation*}
	This shows
	\begin{equation}\label{I_1}
	|I_1|\leq C\varepsilon^{-d}\sup\limits_{i=1,\cdots,N}|X_t^{i,\varepsilon}-Y_t^{i,\varepsilon}|\,.
	\end{equation}
	Next, we rewrite $I_2$ as
	\begin{equation*}
	I_2=\frac{1}{N-1}\sum\limits_{j=2}^NZ_j\,,
	\end{equation*}
	with 
	\begin{equation*}
	Z_j:=F_\varepsilon(Y_t^{1,\varepsilon}-Y_t^{j,\varepsilon})-\int_DF_\varepsilon(Y_t^{1,\varepsilon}-y)\rho_t^\varepsilon(y)\dy \,,
	\qquad j \geq 2 \,.
	\end{equation*}
	Since $\{Y_t^{j,\varepsilon}\}_{j=1}^N$ are i.i.d. random variables, it holds that 
	\begin{equation*}
	\EE[Z_jZ_k]=0 \,,
	\qquad j \neq k \,.
	\end{equation*}
	To see this, we consider $Y_t^{1,\varepsilon}$ is given and denote 
	\begin{equation*}
	\EE'[\cdot]=\EE[\cdot| Y_t^{1,\varepsilon}]\,.
	\end{equation*}
	Since $Z_j$ and $Z_k$ are independent for $j \neq k$, it holds that ${\EE}'[Z_jZ_k]=\EE'[Z_j]\EE'[Z_k]=0$ when $j\neq  k$. Hence
	\begin{equation*}
	\EE[Z_jZ_k]={\EE}_1{\EE}'[Z_jZ_k]=0\,,
	\qquad j \neq k \, ,
	\end{equation*}
	where ${\EE}_1$ means taking expectation on $Y_t^{1,\varepsilon}$.
	
	Hence, one concludes that
	\begin{align}\label{I_2}
	\EE[|I_2|^2]
	&\leq\frac{1}{N-1}\EE[|Z_2|^2] 
	\leq \frac{2}{N-1}\EE\left[F_\varepsilon^2(Y_t^{1,\varepsilon}-Y_t^{2,\varepsilon})+\left(\int_DF_\varepsilon(Y_t^{1,\varepsilon}-y)\rho_t^\varepsilon(y)\dy\right)^2\right]
	\leq \frac{C\varepsilon^{-2(d-1)}}{N-1}\,,
	\end{align}
	where we have used Lemma \ref{lmkenerl} $(b)$ and Lemma \ref{L1} $(a)$.
	Combining estimates \eqref{I_1} and \eqref{I_2}, we have
	\begin{align*}
	\left|\frac{1}{N-1}\sum\limits_{j=2}^NF_\varepsilon(X_t^{1,\varepsilon}-X_t^{j,\varepsilon})-\int_DF_\varepsilon(Y_t^{1,\varepsilon}-y)\rho_t^\varepsilon(y)\dy\right|
	\leq C\varepsilon^{-d}\sup\limits_{i=1,\cdots,N}|X_t^{i,\varepsilon}-Y_t^{i,\varepsilon}|+|I_2|\,,
	\end{align*}
	with the bound~\eqref{I_2} for $I_2$.
	The same estimate holds for the case $i=2,\cdots,N$.
\end{proof}

To prove the propagation of chaos, we need to introduce an 
auxiliary  stochastic mean-field dynamics $\{Y_t^i\}_{i=1}^N$ satisfying
\begin{equation}\label{SDE'system}
\left\{\begin{array}{l}
{Y}_t^i={Y}_0^i+\int_{0}^t\int_{D}F({Y}_s^i-y)\rho_s(y) \dy\ds+\sqrt{2\nu}{B}_t^i-\tilde R_t^i\,, \quad t>0\,,\quad i=1,\cdots,N,\\
\tilde R_t^i=\int_0^tn(Y_s^i) \, {\rm d}|\tilde R^i|_s,\quad |\tilde R^i|_t=\int_0^t\textbf{1}_{\partial D}(Y_s^i) \, {\rm d}|\tilde R^i|_s\,,\\
\end{array}\right.
\end{equation}
where $(\rho_t)_{t\geq0}$ is the common marginal density of $(\{{Y}_t^i\}_{i=1}^N)_{t\geq0}$ for any $t\geq 0$.  Here we set the initial date $(Y_0^i)_{i=1,\cdots, N}=(X_0^i)_{i=1,\cdots, N}$ i.i.d. sharing the same density $\rho_0(x)$. In other words here $(\{{Y}_t^i\}_{i=1}^N)_{t\geq0}$ are $N$ copies of the strong solutions to \eqref{SDE'}.  Similarly let $\{Y_t^{i,\varepsilon}\}_{i=1}^N$ be $N$ copies of the strong solutions to the regularized mean-field dynamics \eqref{RSDE}:
\begin{align} \label{RSDE'system}
\begin{cases}
{Y}_t^{i,\varepsilon}
={Y}_0^i + \int_{0}^t\int_{D}F_\varepsilon({Y}_s^{i,\varepsilon}-y)\rho_s^\varepsilon(y) \dy\ds
+\sqrt{2\nu}{B}_t ^i- \tilde R_t^{i,\varepsilon}, & t>0 \,,\quad i=1,\cdots,N,
\\[2pt]
\tilde R_t^{i,\varepsilon}
= \int_0^t n(Y_s^{i,\varepsilon}) {\, \rm d} |\tilde R^{i,\varepsilon}|_s \,, 
\quad
|\tilde R^{i,\varepsilon}|_t
=\int_0^t\textbf{1}_{\partial D}(Y_s^{i,\varepsilon}) {\, \rm d}|\tilde R^{i,\varepsilon}|_s \,.
\end{cases}
\end{align}

Finally, we present the proof of Theroem~\ref{mainthm2} regarding the propagation of chaos.
\begin{proof}[Proof of Theorem \ref{mainthm2}]
	It follows from Theorem \ref{choithm}, Theorem \ref{mainthm} and Theorem \ref{mainthm1} that there exists a weak solution $\{X_t^{i,\varepsilon}\}_{i=1}^N$ to \eqref{Rparticle} and a unique pathwise solution $\{Y_t^{i}\}_{i=1}^N$ to \eqref{SDE'system} with marginal density $\rho_t$ on the time interval $[0,T]$ for some $T > 0$, such that $\rho_t\in L^\infty\left(0, T; L^\infty(D)\right)$. This implies that we are able to define solutions for those two equations on the same probability space with the same initial condition and Brownian motion. 
	Then it follows from the same argument as in the proof of \eqref{part2}, we obtain that there exists constants $C_T$, $\varepsilon_0$ depending only on $T$, $\norm{\rho^\varepsilon}_{L^\infty\left(0, T; L^\infty(D)\right)}$ and $\norm{\rho}_{L^\infty\left(0, T; L^\infty(D)\right)}$ such that if $\varepsilon<\varepsilon_0$, then for any $1\leq i\leq N $
	\begin{equation*}
	\sup\limits_{t\in[0,T]}\PP\mbox{-ess sup }|{Y}^{i,\varepsilon}_t-Y_t^i| \leq C_T\omega(\varepsilon)^{\frac{1}{2}e^{-CT}}.
	\end{equation*}
	
	Since the  Brownian motions in \eqref{SDE'system} and \eqref{RSDE'system} are prescribed, we can assume $B_t^i=B_t^{i,\varepsilon}$, where $\{B_t^{i,\varepsilon}\}_{i=1}^N$ are used in \eqref{Rparticle}. 
	Using the similar argument as in the proof of Theorem \ref{mainthm}, it follows from It\^{o}'s formula that
	\begin{align*}
	|X_t^{i,\varepsilon}-Y_t^{i,\varepsilon}|^2&=2\int_0^t\left\la X_s^{i,\varepsilon}-Y_s^{i,\varepsilon}, \frac{1}{N-1}\sum\limits_{j\neq i}^NF_\varepsilon(X_s^{i,\varepsilon}-X_s^{j,\varepsilon})-\int_DF_\varepsilon(Y_s^{i,\varepsilon}-y)\rho_t^\varepsilon(y) \dy\right\ra \ds\notag\\
	&\quad-2\int_0^t\la X_s^{i,\varepsilon}-Y_s^{i,\varepsilon}, n(X_s^{i,\varepsilon})\ra {\, \rm d}|R^{i,\varepsilon}|_s-2\int_0^t\la Y_s^{i,\varepsilon}-X_s^{i,\varepsilon}, n(Y_s^{i,\varepsilon})\ra {\, \rm d}|\tilde R^{i,\varepsilon}|_s \notag \\
	&\leq 2\int_0^t\left \la X_s^{i,\varepsilon}-Y_s^{i,\varepsilon}, \frac{1}{N-1}\sum\limits_{j\neq i}^NF_\varepsilon(X_s^{i,\varepsilon}-X_s^{j,\varepsilon})-\int_DF_\varepsilon(Y_s^{i,\varepsilon}-y)\rho_t^\varepsilon(y) \dy\right\ra \ds\,,
	\end{align*}
	where in the second inequality we have used the convexity of the domain $D$. Applying Lemma \ref{lmconsist}, one finds
	\begin{align}\label{ito4}
	|X_t^{i,\varepsilon}-Y_t^{i,\varepsilon}|^2&\leq 2\int_0^t |X_s^{i,\varepsilon}-Y_s^{i,\varepsilon}| \left(C\varepsilon^{-d}\sup\limits_{i=1,\cdots,N}|X_s^{i,\varepsilon}-Y_s^{i,\varepsilon}|+\mathcal{H}_N\right) \ds\notag \\
	&\leq C\varepsilon^{-d}\int_0^t\sup\limits_{i=1,\cdots,N}|X_s^{i,\varepsilon}-Y_s^{i,\varepsilon}|^2 \ds+C\mathcal{H}_NT\,.
	\end{align}
	Taking  the expectation on \eqref{ito4} it gives
	\begin{equation*}
	\EE\left[\sup\limits_{i=1,\cdots,N}|X_t^{i,\varepsilon}-Y_t^{i,\varepsilon}|^2\right]
	\leq 
	C\varepsilon^{-d}\int_0^t\EE\left[\sup\limits_{i=1,\cdots,N}|X_s^{i,\varepsilon}-Y_s^{i,\varepsilon}|^2\right] \ds+\frac{CT\varepsilon^{-(d-1)}}{\sqrt{N-1}}\,,
	\end{equation*}
	where we have used 
	\begin{equation*}
	\EE[|\mathcal{H}_N|]\leq (\EE[|\mathcal{H}_N|^2])^{\frac{1}{2}}\leq  \frac{C\varepsilon^{-(d-1)}}{\sqrt{N-1}}.
	\end{equation*}
	Applying Gronwall's inequality, we have
	\begin{equation*}
	\EE\left[\sup\limits_{i=1,\cdots,N}|X_t^{i,\varepsilon}-Y_t^{i,\varepsilon}|\right]\leq\left(\EE\left[\sup\limits_{i=1,\cdots,N}|X_t^{i,\varepsilon}-Y_t^{i,\varepsilon}|^2\right]\right)^\frac{1}{2}\leq \frac{C_T\varepsilon^{-d+\frac{1}{2}}\exp(C_T\varepsilon^{-d})}{(N-1)^{\frac{1}{4}}}\,,
	\end{equation*}
	for all $t\in[0,T]$. Let $\varepsilon=(\frac{\log N}{8C_T})^{-\frac{1}{d}}$, then one has
	\begin{align*}
	\EE\left[\sup\limits_{i=1,\cdots,N}|X_t^{i,\varepsilon}-Y_t^{i,\varepsilon}|\right]&\leq\left(\EE\left[\sup\limits_{i=1,\cdots,N}|X_t^{i,\varepsilon}-Y_t^{i,\varepsilon}|^2\right]\right)^\frac{1}{2} 
	\leq \frac{N^{\frac{1}{8}}\log (N)}{8C_T(N-1)^{\frac{1}{4}}}\leq CN^{-\frac{1}{8}}\log (N)\,.
	\end{align*}
	
	Now recall the definition of the empirical measures
	\begin{equation*}
	\mu_t^{X,\varepsilon}=\frac{1}{N}\sum\limits_{i=1}^N\delta_{X_t^{i,\varepsilon}}\,,
	\qquad \mu_t^{Y,\varepsilon}=\frac{1}{N}\sum\limits_{i=1}^N\delta_{Y_t^{i,\varepsilon}}\,,
	\qquad \mu_t^Y=\frac{1}{N}\sum\limits_{i=1}^N\delta_{Y_t^i}\,,
	\end{equation*}
	and notice that
	\begin{equation*}
	\pi_\ast=\frac{1}{N}\sum_{i=1}^{N}\delta_{X_t^{i,\varepsilon}}\delta_{Y_t^{i,\varepsilon}}\in \Lambda(\mu_t^{X,\varepsilon},\mu_t^{Y,\varepsilon})\,.
	\end{equation*}
	Then by the definition of $\mathcal{W}_\infty$ distance one has 
	\begin{equation*}
	\mathcal{W}_\infty(\mu_t^{X,\varepsilon},\mu_t^{Y,\varepsilon})\leq \pi_\ast\mathop{\mbox{-ess sup }}\limits_{(x,y)\in D\times D}|x-y|\leq\sup\limits_{i=1,\cdots,N}|X_t^{i,\varepsilon}-Y_t^{i,\varepsilon}| \,.
	\end{equation*}
	This leads to
	\begin{equation}\label{X-Y}
	\EE\left[\mathcal{W}_\infty(\mu_t^{X,\varepsilon},\mu_t^{Y,\varepsilon})\right]\leq \EE\left[\sup\limits_{i=1,\cdots,N}|X_t^{i,\varepsilon}-Y_t^{i,\varepsilon}|\right]\leq CN^{-\frac{1}{8}}\log (N)
	\end{equation}
	and
	\begin{equation}\label{X-Y'}
	\EE\left[\mathcal{W}_\infty(\mu_t^{Y,\varepsilon},\mu_t^{Y})\right]\leq\EE\left[\sup\limits_{i=1,\cdots,N}|Y_t^{i,\varepsilon}-Y_t^{i}|\right]\leq C_T\omega(\varepsilon)^{e^{-CT}}\leq C[\omega(\log^{-\frac{1}{d}}(N))]^{\exp(-CT)}\,.
	\end{equation}
	Furthermore, it follows from the the convergence estimate obtained in \cite[Theorem 1]{fournier2015rate} that
	\begin{equation*}
	\EE\left[\mathcal{W}_2^2(\mu_t^Y,\rho_t)\right]\leq C\left\{\begin{aligned}
	&N^{-\frac{1}{2}}+N^{-\frac{q-2}{q}}, ~&&\mbox{if }d<4\mbox{ and }q\neq 4\,,\\
	&N^{-\frac{1}{2}}\log(N+1)+N^{-\frac{q-2}{q}},~ &&\mbox{if }d=4\mbox{ and }q\neq 4\,,\\
	&N^{-\frac{2}{d}}+N^{-\frac{q-2}{q}}, ~ &&\mbox{if }d>4\mbox{ and }q\neq \frac{d}{d-2}\,.\\
	\end{aligned}\right.
	\end{equation*}
	Combining inequalities \eqref{X-Y} and \eqref{X-Y'}, we arrive at 
	\begin{align*}
	&\EE\left[\mathcal{W}_2(\mu_t^{X,\varepsilon},\rho_t)\right]\leq \EE\left[\mathcal{W}_2(\mu_t^{X,\varepsilon},\mu_t^{Y,\varepsilon})\right]+\EE\left[\mathcal{W}_2(\mu_t^{Y,\varepsilon},\mu_t^Y)\right]+\EE\left[\mathcal{W}_2(\mu_t^Y,\rho_t)\right]\notag\\
	\leq& \EE\left[\mathcal{W}_\infty(\mu_t^{X,\varepsilon},\mu_t^{Y,\varepsilon})\right]+\EE\left[\mathcal{W}_\infty(\mu_t^{Y,\varepsilon},\mu_t^Y)\right]+\EE\left[\mathcal{W}_2(\mu_t^Y,\rho_t)\right]\notag\\
	\leq &CN^{-\frac{1}{8}}(\log N)+C[\omega(\log^{-\frac{1}{d}}(N))]^{\exp(-CT)}\notag\\
	&+C\left\{\begin{aligned}
	&N^{-\frac{1}{4}}+N^{-\frac{q-2}{2q}}, ~&&\mbox{if }d<4\mbox{ and }q\neq 4,\\
	&N^{-\frac{1}{4}}(\log(N+1))^{\frac{1}{2}}+N^{-\frac{q-2}{2q}},~ &&\mbox{if }d=4\mbox{ and }q\neq 4,\\
	&N^{-\frac{1}{d}}+N^{-\frac{q-2}{2q}}, ~ &&\mbox{if }d>4\mbox{ and }q\neq \frac{d}{d-2},\\
	\end{aligned}\right. \notag \\
	\leq&C[\omega(\log^{-\frac{1}{d}}(N))]^{\exp(-CT)}\,,
	\end{align*}
	which completes the proof.
\end{proof}

{\bf Acknowledgments:}
This research was funded in part by Canada’s National Science
and Engineering Research Council (NSERC) through grants to RF
and WS. HH also wants to acknowledge support from National Nature Science Foundation of China (NSFC) (Grant No. 11771237).

\bibliographystyle{amsxport}
\bibliography{bounded}


\end{document}